\documentclass[11pt]{amsart}

\usepackage[english]{babel}

\usepackage[latin1]{inputenc}

 \usepackage[T1]{fontenc}
 
 \usepackage{lmodern}

\usepackage{amssymb}

\usepackage[cmtip,all]{xy}

\usepackage{graphicx}

\usepackage{enumitem}

\newcommand{\overbar}[1]{\mkern 1.5mu\overline{\mkern-1.5mu#1\mkern0mu}\mkern 1.5mu}
 
 \allowdisplaybreaks 
 
\newcommand{\Ad}{\mathrm{Ad}}
\newcommand{\ad}{\mathrm{ad}}
\newcommand{\Aut}{\mathrm{Aut}\,}
\newcommand{\codim}{\mathop{\mathrm{codim\,}}}
\newcommand{\core}{\mathop{\mathrm{Core}}\nolimits}
\newcommand{\dd}{\mathrm{d}}
\newcommand{\Diff}{\mathrm{Diff}}

\newcommand{\eqdef}{\colon\!\!\!=}
\newcommand{\F}{\mathcal{F}}
\newcommand{\g}{\mathfrak{g}}
\newcommand{\GA}{\mathrm{GA}}
\newcommand{\GL}{\mathrm{GL}}

\newcommand{\id}{\mathrm{id}}
\newcommand{\im}{\mathop{\mathrm{im}}}
\newcommand{\inv}{^{-1}}
\newcommand{\invar}{\mathrm{inv}}
\newcommand{\kk}{\mathfrak{k}}
\newcommand{\N}{\mathbb{N}}

\newcommand{\p}{\mathfrak{p}}
\newcommand{\pt}{\mathrm{pt}}
\newcommand{\R}{\mathbb{R}}
 \newcommand{\SL}{\mathrm{SL}}
\newcommand{\trace}{\mathop{\mathrm{trace}}}
\newcommand{\Z}{\mathbb{Z}}

\newtheorem{teo}{Theorem}[section]
\newtheorem{theorem}[teo]{Theorem}
\newtheorem{prop}[teo]{Proposition}
\newtheorem{proposition}[teo]{Proposition}
\newtheorem{lemma}[teo]{Lemma}
\newtheorem{lema}[teo]{Lemma}
\newtheorem{cor}[teo]{Corollary}

\newtheorem{definition}[teo]{Definition}
\newtheorem{defi}[teo]{Definition}
\newtheorem{ejem}[teo]{Example}
\newtheorem{example}[teo]{Example}

\newtheorem{remark}[teo]{Remark}

\begin{document}


\title[Non-unimodular foliations]{Non-unimodular transversely homogeneous foliations}

\author[E.~Mac\'ias]{{E.}~{Mac\'ias-Virg\'os}}
\author[P.L. Mart\'{\i}n]{{P.L.}~{Mart\'{\i}n-M\'endez}}

\address{Departamento de Matem\'aticas,\\ Universidade de San\-tia\-go de Compostela, Spain.}

\email{{\tt quique.macias@usc.es}}
\email{{\tt plmartin@edu.xunta.es}}

\thanks{The first author was partially supported by MINECO Spain Research Project MTM2016-78647-P and by Xunta de Galicia ED431C 2019/10 with FEDER funds}


	
\date{}




\begin{abstract}
We give sufficient conditions for the tautness of a transversely  homogenous foliation defined on a compact manifold, by computing its base-like cohomology. As an application, we prove that if the foliation is non-unimodular then either the ambient manifold, the closure of the leaves or the total space of an associated principal bundle fiber over $S^1$.
\end{abstract}


\maketitle

\setcounter{tocdepth}{2}

\section{Introduction}
A foliation $\F$ on a manifold $M$ is {\em transversely homogeneous} if its transverse holonomy pseudogroup  is generated by the left action of a Lie group $G$ on a  homogeneous space $N=G/K$. Reference \cite{AlvarezNozawa2012} by \'Alvarez and Nozawa contais many examples of this type of foliations.

The fine structure of a transversely homogeneous foliation was established by R. Blumental in his Ph.D. thesis \cite{Blumenthal1978,Blumenthal1979}, and it is described in Theorem \ref{desth}. It can be summarized as follows: there is a {\em holonomy homomorphism} $h\colon \pi_1(M) \to G_\sharp$ (we denote by $G_\sharp$ the  quotient  of the Lie group $G$ by its ineffective subgroup). Let $\Gamma$ be the image of $h$ and let $p\colon \widetilde M \to M$  be the covering of $M$ with fundamental group $\ker h$. Then the induced foliation $p^*\F$  is given by an $h$-equivariant submersion $f\colon \widetilde M \to G/K$, called a {\em developing map} for  $\F$.  This structure theorem will be the main tool in this paper.

When the holonomy pseudogroup $\Gamma$  preserves an invariant metric, the foliation is  a {\em Riemannian foliation}. This condition is ensured, for instance, by asking the isotropy group $K_\sharp$ to be compact. 

In the first part of the paper we are interested in computing the so-called {\em basic} or {\em base-like cohomology} $H(M/\F)$ of the foliation.
Base-like cohomology of a foliation was first introduced by Reinhart  \cite{Reinhart1959} and has been intensively studied since then.  
 
The foliation is {\em unimodular} if  the top-dimensional basic
cohomology group,
$H^q(M(\F)$, $q=\codim \F$, is not null.
In his Ph.D. thesis \cite{Carriere1982} Carri\`ere conjectured that, for Riemannian foliations on
compact manifolds,  being unimodular  is equivalent to being   {\em taut}, the latter meaning that   there exists a Riemannian  metric on $M$
making all leaves minimal submanifolds. This strong result was finally proved by Masa \cite{Masa1992}. A historical account of these results and their importance can be found in \cite{Royo2009}. 

In \cite{Carriere1982},  Carri\`ere also  gave the first example of a Riemannian non-uni\-mo\-du\-lar foliation (see Example \ref{CARREXAMPLE}), which  is in fact a {\em Lie foliation}. Lie foliations are the simplest examples of transversely homogeneous foliations, where $K=\{e\}$ is the trivial subgroup; in other words,   they are transversely modeled on a Lie group with translations as transition maps. In particular, a Lie foliation is necessarily Riemannian. For these foliations it happens that $H(M/\F)$ equals $H_\Gamma(G)$, the cohomology of $\Gamma$-invariants forms on $G$.  
El Kacimi Alaoui and Nicolau   proved the following characterization of unimodular Lie foliations:
\begin{teo}\cite[Theorem 1.2.4]{ElKacimi1990}\label{gkunim}
Let $\overline\Gamma$ be the closure of $\Gamma$ in $G$. Assume that the homogeneous space $\overline \Gamma\backslash G$ is compact and that the groups $G$ and $\overline\Gamma$ are unimodular.  Then $H^n_\Gamma(G)\neq 0$,  where $n=\dim G$.
\end{teo}
The proof is based on the injectivity   of the morphism $i^*\colon H(\g) \to H_\Gamma(G)$ induced by the inclusion $\Omega_G(G) \subset \Omega_\Gamma(G)$. In general, $i^*$ is not injective, as proved by the same authors
\cite[Example 3.2]{ElKacimi1990}.

For general transversely homogeneous foliations, Blumenthal \cite{Blumenthal1980} proved (under some hypothesis) that $H(M/\F)$ equals $H_\Gamma(G/K)$ (see Theorem \ref{cohth}). 

Recall that a Lie group $G$ is {\em unimodular} if its modular function satisfies $\vert m_G\vert=1$ .We shall introduce a related definition (see subsection \ref{UNIMGRUPO}): the Lie group
$G$ is {\em strongly unimodular} if $m_G=1$.
We generalize El Kacimi-Nicolau's result above, by proving:
\begin{teo} \label{MAININTRO}
Assume that $W=\overline\Gamma\backslash G_\sharp$ is compact, the Lie group $G_\sharp$ is strongly unimodular, the subgroup $\overline\Gamma$ is unimodular and $H^q_{G_\sharp}(N)\neq 0$, for  $q=\dim N$. Then $H^q_\Gamma(N)\neq 0$.
\end{teo}
This time,  the proof will rely on the injectivity of the morphism 
$i^*$ induced in cohomology by the inclusion $\Omega_G(G/K)\subset \Omega_\Gamma(G/K)$ (see Theorem \ref{morfismo}). 
 
The second part of the paper exploits those cohomological results. Carri\`ere's example
cited above is defined on a $3$-dimensional manifold $T_A^3$ which is a torus bundle over $S^1$ and the closures of the leaves are tori. We shall prove that this is a general situation in the following Theorem:

 \begin{teo} \label{generalizaciondelrendus}
Let $N=G_0/K_0$ be a homogeneous space, with $G_0$ connected and $(K_0)_\sharp$ compact and strongly unimodular. Let $\F$ be an $N$-transversely homogeneous foliation on the compact manifold  $M$, defined by a developing map whose fibers have a finite number of connected components.   If the foliation $\F$ is not unimodular,  then either $M$, or the closures of the leaves, or the total space of the Blumenthal bundle, fiber over $S^1$.
\end{teo}

As explained in Section \ref{blubundle}, what we call {\em the Blumenthal's fiber bundle} of $\F$ is an auxiliary construction which was defined in \cite{Blumenthal1979} and later studied   in \cite{ElKacimi2001} and \cite{AlvarezNozawa2012}.  It is is a principal $K$-bundle $\overline\rho\colon f^*(G) \to M$, and we prove that its total space is   endowed with a Lie foliation that projects onto the transversely homogeneous foliation $\F$.
If $\F$ is a Lie foliation then $\overline\rho$ is the identity.

The proof of Theorem \ref{generalizaciondelrendus} depends on Tischler's theorem \cite{Tischler1970} about foliations defined by a non-singular closed $1$-form $\omega$ on a compact  manifold. This corresponds to a Lie foliation with $G=\R$, and it happens that $p^*\omega=df$, for the developing map $f$, while the holonomy group $\Gamma$ is the group of periods of the form $\omega$. By deforming $\omega$, Tischler proved that this group turns to be discrete and the manifold $M$ fibers over $S^1$. 

In fact, Theorem \ref{generalizaciondelrendus} can be reformulated as follows: assume that the Lie group $G$ is connected and that the foliation is not unimodular. Also assume that the manifold $M$ and the isotropy group   $K$ are compact. Then either $G$ or $\overline\Gamma$ are not unimodular.
Essentially, we shall use the modular functions of these Lie groups to construct the form $\omega$.

The contents of the paper are as follows. Section \ref{PRELIM} contains preliminaries about homogeneous spaces and unimodular groups, mainly in order to fix our notations. Section \ref{seccionthf} is about the basic definition and properties of 
transversely homogeneous foliations. The main result is Blumenthal's structure theorem, but we state it without assuming that the action of the Lie group $G$ on the manifold $N=G/K$ is effective. This will allow us to do later some explicit constructions using the universal covering  $\widehat N$ of $N$. We also introduce the so-called Blumenthal's fiber bundle, and we discuss the basic notions of Riemannian foliations.

Section \ref{SECTCOHOMOLOGY} is devoted to the relationships between the relative Lie algebra cohomology of the pair $(G,K)$, De Rham cohomology of invariant forms on $N=G/K$ and the base-like  cohomology of the foliation $\F$, including Poincar\'e duality. The main technical result is the injectivity  result in Theorem \ref{morfismo}. Then we prove the main Theorem \ref{MAININTRO} and we give an example with $G=\SL(n,\R)$. 

In the first part of the paper  we do not assume that the group $G$ is connected.
But the results are limited to transversely homogeneous foliations where the developing map has connected fibers. 

In order to go further, we introduce in Section \ref{NONCONN} what we call the ``extended group''.
It is the smallest group containing $G$ such that  the original $N$-transversely homogeneous foliation can be given a structure of $\widehat N$-transversely homogeneous foliation. This construction is less restrictive than a similar one in Blumenthal's paper \cite{Blumenthal1979}, where he considers the whole group of isometries of $N$. We also need to reformulate Blumenthal results \cite[Theorem 3.ii)]{Blumenthal1979} about the closure $\overline L$ of each leaf $L$, in such a way that the foliation $\F$, when restricted to $\overline L$, is also a transversely homogeneous foliation modeled by a homogeneous space where the group that acts transitively is a subgroup of the extended group.

By applying the results of the first part of the paper to this new foliated structure, we are able to prove  in all generality the characterization of unimodular foliations (Theorem \ref{tercerteoremadeuni}) and Theorem \ref{generalizaciondelrendus} cited above. These results generalize analogous results for Lie foliations that we announced in \cite{Macias2005}.

\begin{remark}About notation:
 in the first part of the paper the Lie group $G$ may not be connected. In the second part, we denote by $G_0$ a {\em connected} Lie group, while $G$ will be an ``extended'' group to which the results of the first part apply.
 \end{remark}


\section{Preliminaries}\label{PRELIM}
In order to fix our notations,
we recall several previous results about Lie groups and homogeneous spaces.
\subsection{Homogeneous spaces} 
Let $G$ be a Lie group, which is not supposed to be neither connected nor simply connected. 
Assume that  $G$ acts transitively on the {\em connected} manifold $N$.  Fix a base point $o\in N$, and denote by $K$ the isotropy group $G_o$, so the map $[g]\in G/K \mapsto g\cdot o\in N$ is a diffeomorphism of $G$-spaces. 

\begin{remark}It can be proved that $G_e$, the connected component of the identity, also acts transitively on $N$. However, in Section \ref{seccionthf} we shall need a non-connected Lie group with an additional condition that $G_e$ does not fulfill. 
 \end{remark}
 
 For $g\in G$ we denote by $\lambda(g)\colon N \to N$ the left translation $\lambda(g)(p)=g\cdot p$.

\begin{definition}\cite{Scott1987} \label{defcore}The {\em normal core} of the action is the kernel, denoted by $\core(K)$, of the morphism $\lambda\colon G \to \Diff(N)$, that is, 
$$\core(K) =\{g\in G \colon \lambda(g)=\id\}.$$
\end{definition}

Notice that the action of $G$ on $N$ is  effective if and only if $\core(K)=\{e\}$.
We list here some properties of the normal core. The proof is easy.

\begin{prop} \label{caractcore}The normal core $\core(K)$ equals:
\begin{enumerate}
\item
the intersection $\bigcap\nolimits_{p\in N} G_{p}$ of the isotropy subgroups.
\item
the intersection $\bigcap\nolimits_{g\in G}gKg\inv$,  of the conjugate subgroups.
\item
the set $\{ k\in K\colon gkg^{-1}\in K \; \forall \, g\in G\}$.
\end{enumerate}
\end{prop}

It follows that  
$\core(K)$ is the largest subgroup of $K$ which is normal in $G$. 
We denote by 
$G_\sharp=G/\core(K)$
the quotient group.

\begin{prop} \label{isotropiagks}
 The induced action of $G_\sharp$ on $N$ is effective, with isotropy $K_\sharp=K/\core(K)$. Hence $N$ is diffeomorphic to $G_\sharp/K_\sharp$.
\end{prop}

\subsection{Unimodular groups}\label{UNIMGRUPO}
Let $\g$ be the Lie algebra of the Lie group $G$. 

\begin{defi}
The  Lie algebra $\g$ is {\em unimodular} if $\trace\ad_X=0$ for all $X\in \g$.
\end{defi}

Every Lie group $G$ admits a non-zero  left invariant measure $\mu$, which is called {\em a Haar measure}. It is unique up to a positive factor.  See for instance \cite{Faraut2008}.

\begin{defi} \label{funcmodular}
The {\em modular function} of  $G$ is the Lie group morphism $m_G\colon G\to (\R^+,\cdot)$ given by
$\mu(Eg)=m_G(g)\mu(E)$
for every Borel set $E\subset G$.
\end{defi}

We say that the group $G$  is {\em unimodular} if $m_G\equiv 1$. Equivalently, the Haar measure is bi-invariant.
\begin{example} Every discrete (or abelian, or compact) Lie group is unimodular.
\end{example}

When $\dim G\geq 1$, Definition \ref{funcmodular} is equivalent to the following one:
\begin{defi} \label{seccionfuncmodular}
The  modular function is given by
$$
m_G(g)=\vert \det \Ad_G(g)\vert.
$$
\end{defi}

We introduce  a new definition, that we shall need later as an hypothesis.
\begin{defi} \label{deffuerteuni}
The Lie group $G$ is  {\em strongly   unimodular} if  $$\det \Ad_G(g)=1,\quad \text{for all\ } g\in G.$$
\end{defi}

Obviously, any {\em connected} unimodular Lie group is strongly unimodular.

\begin{prop} \label{geunimodular}\ 
\begin{enumerate}
\item
The Lie algebra $\g$ is unimodular if and only if the the connected component $G_e$ of the identity   is unimodular.
\item
If the Lie group $G$ is unimodular then $G_e$ is unimodular.
\end{enumerate}
\end{prop}

\begin{ejem}
The converse is not true when $G$ is not connected. 
For instance, for a fixed $\lambda\in(0,\infty)$, consider  the subgroup of  $\SL(2,\R)$ defined as
$$
G=\{\begin{bmatrix}\lambda^n&t\cr 0&\lambda^{-n}\end{bmatrix}\colon n\in\Z,\; t\in\R\}.
$$
The modular function  is $m_G(n,t)=\lambda^{2n}$, so, in general, the Lie group $G$ is not unimodular. However, its connected component $G_e=\R$ is unimodular.
\end{ejem}

\begin{prop}\label{funcionmodulardee2}
For a covering $p\colon G\to G^\prime$ of Lie groups (that is, a surjective Lie group morphism with discrete kernel),  we have 
$$\det \Ad_G(g)=\det \Ad_{G^\prime}(p(g)), \quad \forall  g\in G.$$
\end{prop}

\section{Transversely homogeneous foliations} \label{seccionthf}
In this section we give the fundamental definitions and results about transversely homogeneous foliations.  The main ``structure theorem'' \ref{desth} is due to Blumenthal  \cite{Blumenthal1979}, which stated it   when the action of $G$ on $N$ is effective.   

\subsection{Structure theorem}\label{stth}

Let $N=G/K$ be a {\em connected} $G$-homogeneous space (the Lie group may not be connected)  and let $(M,\F)$ be a foliated differentiable manifold.

\begin{definition} \label{subth} The foliation $\F$ on $M$ is {\em transversely homogeneous} with transverse model  $N$  if it is defined by a family of submersions
$f_\alpha\colon  U_\alpha\subset M \to  N$,  which satisfy:
\begin{enumerate}
\item
$\{ U_\alpha \}$ is an open covering of $M$;
\item
if $U_\alpha\cap U_\beta\not =\emptyset$ then $f_\alpha=\lambda(g_{\alpha\beta})\circ f_\beta$ on $f_\alpha\inv(U_\alpha\cap U_\beta)$, for some $g_{\alpha\beta}\in G$.
\end{enumerate}
\end{definition}

\begin{theorem}[Structure theorem \cite{Blumenthal1979}]
\label{desth} 
Let $\F$ be a transversely homogeneous foliation on the manifold $M$. There exists a regular covering   $p\colon \widetilde M\to M$ such that
\begin{enumerate}
\item
the automorphism group $\Aut (p)$ of the covering is isomorphic to a subgroup  $\Gamma$ of $G_\sharp=G/\core(K)$;
\item
the lifted foliation $\widetilde F=p^*\F$ is the simple foliation $f^*\pt$ associated to some submersion $f\colon \widetilde M\to N$;
\item
the submersion $f$ is {\em equivariant} by the isomorphism $h\colon \Aut (p)\cong \Gamma$, that is, $f(\gamma\cdot\widetilde x)=h(\gamma) f(\widetilde x)$, for all $\widetilde x\in \widetilde M$ and all $\gamma\in \Aut (p)$.
\end{enumerate}
Conversely, if $\F$ is a foliation on $M$ for which there exists a regular covering satisfying the three properties above, then   $\F$ is an transversely homogeneous foliation. 
\end{theorem}

The group $\Gamma$ and the submersion $f$ are called the {\em holonomy group} and the {\em developing map} of the foliation, respectively.

\begin{remark}\label{CONECCTEDFIB}
Sometimes, a larger covering than $\widetilde M$ will be considered, for instance the universal covering. If $\widetilde p\colon \widehat M\to \widetilde M$ is a regular covering, then the composition $f\circ \widetilde p\colon \widehat M \to N$  is equivariant by the epimorphism $$\Aut (p\circ \widetilde p)\to\Aut (p)\cong\Gamma\subset G_\sharp.$$ 

Conversely, if there exist a covering $\widehat p\colon \widehat M\to M$, a morphism $\widehat h\colon\Aut(\widehat p)\to G_\sharp$,  and a $\widehat h$-equivariant submersion $\widehat f\colon \widehat M\to  N$, then, for  the covering $p\colon \widetilde M\to M$ associated to the kernel of $\widehat h$ there is an induced submersion $f\colon \widetilde M\to N$, which is invariant by the induced isomorphism  $ \Aut p\cong \im \widehat h=\Gamma\subset G_\sharp$.

On the other hand, if the developing map $\widehat f\colon \widehat M\to  N$ has connected fibers then $f\colon \widetilde M\to N$ will have connected fibers too. Analogously, if the fibers of $\widehat f$ have a finite number of connected components then so has  $f$.
\end{remark}

\begin{example}\label{liefolprimeraparte} A {\em Lie foliation} \cite{Macias1994} is a transversely homogeneous foliation with  model  $N= G$ a Lie group; that is, the subgroup $K$ is trivial. The structure theorem for Lie foliations was proved by Fedida \cite{Fedida1973}.
\end{example}

\subsection{The Blumenthal bundle} \label{blubundle}
Blumenthal    \cite[\S 3]{Blumenthal1979} introduced a principal  bundle  associated to each transversely homogeneous foliation $\F$  (for an effective action).
El Kacimi, Guasp and Nicolau \cite{ElKacimi2001}  studied this fiber bundle 
for a type of foliations where  that bundle is trivial. 

Taking into account Proposition \ref{isotropiagks},
we consider the pullback of the canonical projection $\pi_\sharp\colon G_\sharp \to N=G_\sharp/K_\sharp$ by the developing map $f\colon \widetilde M\to N$.  That is,
$$f^*(G_\sharp)=\{ (\widetilde x,g)\in \widetilde M\times G_\sharp\colon f(\widetilde x)=\pi_\sharp(g)\}.$$ 
Let  $\rho\colon f^*(G_\sharp) \to \widetilde M$ and  $\overbar f\colon f^*(G_\sharp) \to G_\sharp$ be  the maps   induced by the  projections. We have 
$$\pi_\sharp\circ \overbar f=f\circ \rho.$$

\begin{prop} \label{taucubiertaregular}
\begin{enumerate}
\item
The action of $\Aut p\stackrel{h}\cong \Gamma$ on $f^*(G_\sharp)$, defined by   
$$\gamma\cdot(\widetilde x,g)=(\gamma\cdot \widetilde x,h(\gamma) g)$$
is free, properly discontinuous and transitive on the fibers. 
\item
As a consequence, the projection  
$$\tau\colon f^*(G_\sharp)\to  \Gamma\backslash f^*(G_\sharp)$$ 
onto the orbit space is a regular covering, with deck group $\Gamma$. 
\item
Moreover the map   $\rho\colon f^*(G_\sharp) \to \widetilde M$  is equivariant.
\item
The map   $\overbar\rho\colon \Gamma\backslash f^*(G_\sharp)\to M$ is a principal bundle with structure group  $K_\sharp$. 
\end{enumerate}
\end{prop}

We shall call $\overbar\rho$ it the {\em Blumenthal bundle} of the foliation.

\begin{remark}As pointed out by Blumenthal, the lifted foliation $\overline\rho\,^*(\F)$ on $\Gamma\backslash f^*(G_\sharp)$ equals the projection, by the covering map $\tau$, of the foliation $\overline f\,^*(\F_0)$ on $f^*(G_\sharp)$, where $\F_0$ is the foliation on the Lie group $G_\sharp$ by  (the connected components of) the cosets of the subgroup $K_\sharp$. We have $\codim \overline f\,^*(\F_0)=\codim \F =\dim G/K$.

Notice that there is another foliation on the total space $\Gamma\backslash f^*(G_\sharp)$, namely the projection by $\tau$ of $\overline f^*\pt$. Its codimension equals $\dim G$, so  its dimension equals $\dim\F$. It is a Lie foliation, whose associated transversely homogeneous foliation is $\overline\rho\,^*(\F)$.
\end{remark}
 
\subsection{Riemannian foliations}
In this paragraph we assume that the normal core $K_\sharp$ is compact. 
This assumption 
has important consequences. First, there exists a Riemannian metric on $N$ which is $G_\sharp$-invariant.
It follows that there exists a metric on $M$ which is a bundle-like metric for the foliation $\F$, that is, $\F$ is a Riemannian foliation (see for instance  the proof of Theorem 4.1. in \cite{Blumenthal1979}).
That metric on $M$ lifts to a $\Gamma$-invariant metric on the covering $\widetilde M$, which is a bundle-like metric for the lifted foliation  $\widetilde \F$.
By construction of the metrics above it follows that the developing $p\colon \widetilde M\to M$ is a Riemannian submersion. Then  Hermann's Theorem 1 in \cite{Hermann1960} for Riemannian submersions between complete manifolds applies if $M$ is compact.

\begin{proposition} \label{friemannianasikcompacto}
If $M$ and $K_\sharp$ are compact, then the  developing submersion   $f\colon \widetilde M\to N$ is a locally trivial bundle (in particular, the map is surjective).
\end{proposition}
 
\begin{prop} \label{fblucompacto}
If the manifold $M$   and the group $K_\sharp$ are both compact, then
\begin{enumerate}
\item \label{fblucompacto1}
the total space $\Gamma\backslash f^*(G_\sharp)$ of the Blumenthal bundle is compact;
\item \label{fblucompacto2}
the quotient manifold $W=\overline \Gamma\backslash G_\sharp$ of the Lie group $G_\sharp$ by the closure $\overline \Gamma$ of the holonomy group $\Gamma$ is compact. 
\end{enumerate} 
\end{prop}
\begin{proof}

1)  If the fiber bundle $\overline \rho$ has  compact fibers, then it is a proper map.

2) The hypothesis imply that the developing map $f=\pi \circ \overline f$ is a locally trivial bundle (see Proposition \ref{friemannianasikcompacto} below), hence it is a surjective map. This implies the surjectiveness of $\overline f \colon f^*(G_\sharp ) \to  G_\sharp$. Define
$$
\varphi\colon \Gamma\backslash f^*G_\sharp \to W=\overline\Gamma\backslash G_\sharp,
$$
by 
$$\varphi(x)=[\overline f(\widetilde x)],$$
where  $\widetilde x\in f^*(G_\sharp)$ verifies $\tau(\widetilde x)=x$. This map is well defined and continuous, and it is surjective by the surjectiveness of $\overline f$. Then $W$ is compact.
\end{proof}

\section{Cohomology}\label{SECTCOHOMOLOGY}
In this section we study the relationship between the De Rham invariant cohomology of the homogeneous space $N=G/K=G_\sharp/K_\sharp$ and the Lie algebra cohomology of the reductive pair $(\g,K_\sharp)$, including Poincar\'e duality. We will follow Section VII.9 of Knapp's book \cite{Knapp1988} and Hazewinkel's paper \cite{Hazewinkel1970}, with some slight changes.

\subsection{Relative Lie algebra cohomology}
 As it is well known, when the Lie algebra $\g$ is unimodular its cohomology verifies the Poincar\'e duality $H^r(\g;\R)\cong H^{n-r}(\g;\R)$, for $n=\dim\g$. In our context we need a much more general result about relative cohomology. 
For the sake of completeness we include the basic definitions and results but we will skip the details of the proofs.

\subsubsection{Reductive pairs}

We denote by $\g$ and $\kk$ the Lie algebras of $G_\sharp$ and $K_\sharp$, respectively. 

\begin{defi}
The pair $(\g,K_\sharp)$ is {\em  reductive} if there exist a vector subspace  $\p\subset\g$  such that $\g=\kk\oplus\p$ and $\Ad_{G_\sharp}(k)(\p)\subset \p$, for all  $k\in K_\sharp$.
\end{defi}

When $G_\sharp$ is connected, the last condition is equivalent to  $[\kk,\p]\subset \p$.

\begin{prop}\cite[Proposition 3.16]{Cheeger1975}\label{EFFISOM}
If the action of $G_\sharp$ on $N=G_\sharp/K_\sharp$ is effective and by isometries, then the pair $(\g,K_\sharp)$ is reductive.
\end{prop}

\begin{defi}\cite[page 334]{Knapp1988}\label{defidegkmodulo}
The vector space $V$ is a {\em $(\g,K_\sharp)$-module} if there are representations $\rho\colon\g\to\mathrm{End}(V)$ and $\alpha\colon K_\sharp\to \GL(V)$ verifying the following conditions:
\begin{enumerate}
\item 
the differentiated version of the $K_\sharp$ action is the restriction to $\kk$ of the $\g$ action, that is, $\alpha_*=\rho\mid_\kk$, or equivalently,
$$
 X\cdot v=  ( {\dd}/{\dd t} )(\exp(tX)\cdot v)_{\vert {t=0}}  ,\quad \forall X\in \kk, v\in V;
$$
\item
there is a compatibility condition
$$
(\Ad_{G_\sharp}(k)(X))\cdot v=k\cdot (X\cdot(k^{-1}\cdot v)), \quad  \forall k\in K_\sharp, X\in\g,  v\in V;
$$
\item
the vector space $V$ is $K_\sharp$-finite, that is, 
 $K_\sharp\cdot v$ generates a finite dimensional subspace of $V$,  for all $v\in V$,
\end{enumerate}
where we denote 
$$X\cdot v=\rho(X)(v), \quad 
k\cdot v=\alpha(k)(v), \quad
\mathrm{for\ } k\in K_\sharp, X\in\g, v\in V.$$
\end{defi}

\begin{ejem}The trivial module $V=\R$ is endowed with the actions $X\cdot t=0$ and $k\cdot t=t$.
\end{ejem}

\begin{ejem}If $V$ is a $(\g,K_\sharp)$-module, the dual space $V^*$ can be endowed with the following actions:
\begin{align*}
(X\varphi)(v)&=-\varphi(Xv),\\
(k\varphi)(v)&=\varphi(k\inv v),\quad \mathrm{for\ }X\in\g, k\in K_\sharp, \varphi\in V^*, v\in V.
\end{align*}
Only the subspace $(V^*)_{K_\sharp}$ of $K_\sharp$-finite elements will be a $(\g,K_\sharp)$-module.
\end{ejem}

\begin{ejem}If $V,W$ are $(\g,K_\sharp)$-modules then the tensor space $V\otimes_\R W$ is a $(\g,K_\sharp)$-module with the actions:
\begin{align*}
X(v\otimes w)&=Xv\otimes w + v\otimes Xw,\\
k(v\otimes w)&=kv\otimes kw, \quad \mathrm{for\ } X\in\g, k\in K_\sharp, v\in V, w\in W.
\end{align*}
\end{ejem}

\subsubsection{The Hazewinkel module}

\begin{defi}\cite{Hazewinkel1970}\label{HAZ} Let $V$ be a $(\g,K_\sharp)$-module. Assume that the Lie algebra $\kk$ is unimodular. The {\em Hazewinkel module} $V^{tw}$ is the space $V$ endowed with the actions:
\begin{align*}
X\odot v &= X\cdot v -\trace \ad_X v,\\
k\odot v&=\det\Ad_\p(k)\inv k\cdot v,
\end{align*}
where we denote by $\Ad_\p(k)$ the restriction of $\Ad_G(k)$, $k\in \kk$, to the vector space $\p$.
\end{defi}

\begin{remark}
The Hazewinkel module $V^{tw}$ is a $(\g,K_\sharp)$-module when the trace of $\ad_X$, for $X\in\kk$, equals that of its restriction to $\p$. This is why we need  the trace of the restriction of $\ad_X$ to $\kk$ to be zero, that is, the Lie algebra $\kk$ to be unimodular. 
\end{remark}

\begin{prop}Let $q$ be the dimension of  $\p\cong \g/\kk$. The module $V^{tw}$ is isomorphic to the module $V\otimes_\R (\Lambda^q\p)^*$.
\end{prop} 

The precise definition of the module structure on $\Lambda^q\p$ is given in  \cite[Lemma 7.30]{Knapp1988}.

\subsubsection{Relative cohomology}

Let $V$ be a $(\g,K_\sharp)$-module. Assuming that the pair $(\g,K_\sharp)$ is reductive, the exterior algebra $\Lambda^r\p$, $0\leq r \leq q$,  inherits  a structure of $K_\sharp$-module from the adjoint action on $\p$, so we can consider the cochain complex
 $L_{K_\sharp}(\Lambda^r\p,V)$ of  $\R$-linear maps of $K_\sharp$-modules between $\Lambda^r\p$ and $V$.  

\begin{defi}
The {\em relative  cohomology groups with coefficients in  $V$},
$$H^r(\g,K_\sharp;V)$$  
are 
 the cohomology groups of the complex $L_{K_\sharp}(\Lambda^r\p,V)$.
 \end{defi} 
 The precise definition of these spaces and the differential of the complex can be found in \cite[pages 395--396]{Knapp1988}. 
 
 \begin{example}
For $r=0$, the space $L_{K_\sharp}(\Lambda^0\p,V)$ is isomorphic to the $K_\sharp$-invariant subspace $$V^{K_\sharp}=\{v\in V\colon k\cdot v=v \quad\forall k\in K_\sharp\}$$
and we define $(\delta v)(X)=X\cdot v$. 
\end{example}
 
\begin{ejem} \label{calculodehcero}
$H^0(\g,K_\sharp;V)$ equals $V^{K_\sharp,\p}$,  the space of elements of   $V$ which are invariant by the actions of $K_\sharp$ and $\p$.
\end{ejem}

 Analogously, we can consider homology.
 
\begin{defi}
The  {\em relative homology groups} $H_r(\g,K_\sharp;V)$ are the homology groups of the the chain complex  $\Lambda^r\p\otimes_{K_\sharp} V$.
 \end{defi}
The differential $\partial$ of this complex is defined in \cite[pages 394--395]{Knapp1988}. 

\begin{example}
For $r=1$ we have $
\partial(X\otimes v)=-X\cdot v.$
Also, $\Lambda^0\otimes_{K_\sharp} V=V^{K_\sharp}$, the space of $K_\sharp$-invariant vectors.
\end{example}

\subsubsection{Poincar\'e duality}

\begin{theorem}[Poincar\'e duality, {\cite[Theorem 7.31]{Knapp1988}}]
If the pair $(\g,K_\sharp)$ is reductive and $\kk$ is unimodular (in particular, if $K_\sharp$ is compact) then 
\begin{enumerate}
\item
$H^r(\g,K_\sharp;V^c)\cong H_r(\g,K_\sharp;V)^*$,
\item
$H^r(\g,K_\sharp;V)\cong H_{q-r}(\g,K_\sharp;V^{tw})$,
\end{enumerate}
for $0\leq r \leq q=\dim \p$, where $V^c=(V^*)_{K_\sharp}$ is the set of $K_\sharp$-finite elements of the dual space and   
$V^{tw}$ is the Hazewinkel module.
\end{theorem}

\begin{proof}(sketch)
Taking into account the natural isomorphism of complexes
$$F\colon (\Lambda^r\p\otimes_{K_\sharp} V)^*  \cong L_{K_\sharp}(\Lambda^r\p,V^c)$$
given by $F(a\otimes v)=F(a)(v)$,   we have (1).

 On the other hand, if $\epsilon_0$ is a generator of $\Lambda^q\p\cong\R$ we consider  the isomorphism of complexes 
$$\lambda\colon \Lambda^r\p\otimes_{K_\sharp} V^{tw}\cong L_{K_\sharp}(\Lambda^{q-r}\p, V)$$
given by $\lambda(\alpha\otimes v)(\beta)=\epsilon_0(\alpha\wedge \beta)v$ and we have (2).
\end{proof}

\begin{cor} \label{dualpoincare}
Taking $V=\R$ with the trivial $(\g,K_\sharp)$-module structure, we have
$$
H^q(\g,K_\sharp;\R)^*=H^0(\g,K_\sharp;(\R^{tw})^*),
$$
where $q=\dim N$. 
\end{cor}

Finally, we need the following Lemma.
\begin{lema}\label{PREVIO}
Assume that $K_\sharp$ is unimodular. If $G_\sharp$  and $K_\sharp$ are strongly unimodular, then   $\trace \ad(X)=0$ for all $X\in \p$ and $\det \Ad_\p(k)=1$ for all $k\in K_\sharp$. The converse is true when $G_\sharp$ is connected.
\end{lema}

\begin{proof}
Since $G_\sharp$ is unimodular, its Lie algebra is $\g$ unimodular too, hence  $\trace\ad(X)=0$ for all $X\in \g$. 

On the other hand, the condition $\Ad_{G_\sharp}(k)(\p)\subset \p$ means that the matrix associated  to $\Ad_{G_\sharp}(k)$ has the form $\left [\begin{matrix} *&0\\ 0& *\end{matrix}\right ]$, so
\begin{equation} \label{descadj}
\det \Ad_{G_\sharp}(k)=\det \Ad_\kk(k)\cdot \det \Ad_\p(k)\quad \forall \, k\in K.
\end{equation}
But $\det \Ad_{G_\sharp}(k)=1$ and $\det \Ad_\kk(k)=1$, for all $k\in K_\sharp$, by hypothesis, and the result follows.
\end{proof}

With all that machinery we can prove the following result.

\begin{teo} \label{hcero}
 Let the pair $(\g,K_\sharp)$  be reductive. If the groups $G_\sharp$ and $K_\sharp$ are strongly unimodular then $$H^0(\g,K_\sharp;(\R^{tw})^*)\not =0.$$  
Conversely, assume that $K_\sharp$ is unimodular. If $G_\sharp$ is connected,  the condition $H^0(\g,K_\sharp;(\R^{tw})^*)\not =0$ implies that $G_\sharp$ and $K_\sharp$ are strongly unimodular.
\end{teo}
\begin{proof}
Accordingly to Example \ref{calculodehcero}, the elements of $H^0(\g,K_\sharp;(\R^{tw})^*)$ will be those $\varphi\in (\R^{tw})^*$ which are invariant by the action of  $K_\sharp$ and by the action of $\p$. Let us see what that means:
\begin{enumerate}
\item
We consider on $\p$ the structure dual to that of Hazewinkel. Then
 $$
\varphi(X\cdot v)=\trace \ad(X)\varphi(v),\quad \forall\, X\in \p,\, v\in \R.
$$
But the action of $\p$ on $\R$ is trivial,  so $\varphi(X\cdot v)=\varphi(0)=0$. Hence, an element   $\varphi\not =0$ is invariant if and only if $\trace \ad(X)=0$ for all $X\in \p$, so all the elements of $(\R^{tw})^*$ are invariant for the action of $\p$.
\item
On the other hand, that $\varphi$ is invariant by the action of  $K_\sharp$ means that 
$$k\cdot_t \varphi=\ \det \Ad_\p(k)\varphi(k^{-1}\cdot v),$$
for all $k\in K_\sharp$.
But the action of $K_\sharp$ on $\R$ being trivial, we have $\varphi(k^{-1}\cdot v)=\varphi(v)$, so $\varphi\not =0$ is invariant if and only if  $\det\Ad_\p(k)=1$ for all $k\in K_\sharp$. Again, all the elements of $(\R^{tw})^*$ will be invariant by the action of $K_\sharp$.
\end{enumerate}
Summarizing, either $H^0(\g,K;(\R^{tw})^*)=0$ or $H^0(\g,K;(\R^{tw})^*)=\R^*\cong \R$, and this can happen if and only if $\trace \ad(X)=0$ for all $X\in \p$ and $\det \Ad_\p(k)=1$ for all $k\in K_\sharp$. The result then follows from Lemma \ref{PREVIO}.
\end{proof}

\subsection{De Rham cohomology}
Let $N=G_\sharp/K_\sharp$ be a connected homogeneous space and let $\Gamma\subset G_\sharp$ be a subgroup. We shall denote by $H_\Gamma(N)$ the cohomology of the De Rham complex $\Omega^\bullet_\Gamma(N)$ of differential forms on $N$ which are $\Gamma$-invariant. If  $\overline \Gamma$ is the closure of $\Gamma$ in $G_\sharp$ then $H_{\Gamma}(N)= H_{\overline\Gamma}(N)$.

Our main result in this section is the following one.

\begin{teo} \label{morfismo}Let $i^*\colon H_{G_\sharp}(N)\to H_\Gamma({N})
$
be the morphism induced in cohomology by the inclusion  $ \Omega_{G_\sharp}(N)\subset  \Omega_\Gamma(N)$.
If the manifold $W=\overline\Gamma\backslash G_\sharp$ is compact and  there exists a volume form on $W$ which is right invariant by the action of $G_\sharp$, then $i^*$ is injective.
\end{teo}
\begin{proof}
It is enough to define a morphism of complexes  $r\colon \Omega_{\overline\Gamma}(N) \to \Omega_{G_\sharp}(N)$ such that $r\circ i = \id$.   
Consider the map $\lambda\colon G \to \Diff(N)$ given by $\lambda(g)(p)=g\cdot p$ and define, for each ${\overline\Gamma}$-invariant differential  form $\alpha$ of degree $s$ on $N$, that is,  $\alpha\in \Omega^s_{\overline\Gamma}(N)$, the following map: 
$$\phi_\alpha\colon x=[g]\in W=\overline\Gamma\backslash G_\sharp \mapsto x^*\alpha=\lambda(g)^*\alpha \in \Omega^s(N).$$
It is well-defined because, if $h\in {\overline\Gamma}$ then
$$\lambda(hg)^*\alpha =(\lambda(h)\circ \lambda(g))^*\alpha=\lambda(g)^*\lambda(h)^*\alpha=\lambda(g)^*\alpha.$$

 We denote by 
$$r(\alpha)=\int_W{(x^*\alpha)\, \omega(x)}$$ the following $r$-form on $N$:
$$r(\alpha)_{[g]}(X_1([g]),\dots,X_s([g]))=\int_W{(x^*\alpha)_{[g]}(X_1([g]),\dots,X_s([g]))\, \omega(x)},$$
where $\omega$ is the invariant volume form, that we can assume that verifies
$$\int_W\omega(x)=1.$$
 
Now it is routine to check the following two properties:
\begin{enumerate}
\item
$r(\alpha)$ is $G_\sharp$-invariant
\item
If $\alpha$ is $G_\sharp$-invariant then $r(\alpha)=\alpha$.
\end{enumerate}
Finally, $r$ is a morphism of complexes by the property
\begin{enumerate}\item[(3)]
$r(d\alpha)=dr(\alpha),$
\end{enumerate}
which can be proved taken into account the following result:

\begin{teo}[Derivation under the integral sign]  \label{dbsiv}
Let $W$ and $N$ be two smooth manifolds. Assume that $W$ is compact and orientable. Then, for each smooth function $g\colon W\times N\to {\mathbb R}$ and each smooth vector field $X$ on $N$, we have
$$
\int_W X  g(x,p)  \cdot\omega (x)=X  \int_W g(x,p)\cdot \omega(x).
$$
where the derivation $X$ is relative to the variable $p$.
\end{teo}
\end{proof}

As a Corollary we shall obtain  Theorem \ref{MAININTRO} about the non-nullity of the top cohomology group, that we stated in   the Introduction.
 
\begin{proof}[Proof of Theorem \ref{MAININTRO}]
 
First, assume that $\dim \overline\Gamma>0$. 

Since $G_\sharp$ is strongly unimodular we have $\det \Ad_{G_\sharp}(\gamma)=1$ for all $\gamma\in \overline\Gamma$.  However,   $\overline\Gamma$ may not be connected, so it may happen that $\det \Ad _{\overline\Gamma}(\gamma)=-1$ for some $\gamma$.

If $\det \Ad _{\overline\Gamma}(\gamma)=1=\det \Ad _{G_\sharp}(\gamma)$ for all $\gamma\in \overline\Gamma$, we know from \cite[Proposition 1.6.]{Helgason1962}  that there exists on $W=\overline\Gamma\backslash G_\sharp$ an invariant volume form, which implies, by Theorem \ref{morfismo} that the morphism $H_{G_\sharp}(N) \to H_{\overline\Gamma}(N)$ is injective. 

On the other hand, if $\det \Ad _{\overline\Gamma}(\gamma)=-1$ for some $\gamma$, we can consider, as we did in \cite{Macias2005},  the subgroup $H_2=\{ \gamma\in \overline\Gamma\colon \det \Ad _{\overline\Gamma}(\gamma)>0\}$ and the manifold $W_2=H_2\backslash G_\sharp$. In this way, $W_2$ is compact and  $\det \Ad _{H_2}(h)=1=\det \Ad _{G_\sharp}(h)$ for all $h\in H_2$. Hence, the morphism $H_{G_\sharp}(N) \to H_{H_2}(N)$ is injective, by Theorem \ref{morfismo}. Now, we can consider the composition  
$$
\Omega_{G_\sharp}(N)\to \Omega_{\overline\Gamma}(N)\to \Omega_{H_2}(N)
$$
and the induced morphism $H_{G_\sharp}(N) \to H_{\overline\Gamma}(N)$ will be injective too.

In both cases, taking into account that $H^q_{G_\sharp}(N)\neq 0$, we have $H^q_{\overline\Gamma}(G_\sharp)\not =0$, as stated.

When  $\overline\Gamma$ is a discrete group, we can argue in the following way:  since $G_\sharp$ is unimodular, it admits a bi-invariant volume form $\omega$. Since $G_\sharp \to W=\overline\Gamma\backslash G_\sharp$ is a covering,  $\omega$ induces a form $\overline\omega$ on $W$ which is $G_\sharp$-invariant. Finally, since $W$ is compact, Stokes theorem implies that $\overline\omega$ is a volume form.
\end{proof}

We now recall how to compute  the cohomology of the complex of invariant forms on the homogeneous space $N=G_\sharp/K_\sharp$. If $o=[e]\in N$  we denote $\p=\g/\kk=T_oN$. and by  $\Ad_\p(k)$, with $k\in K_\sharp$, denotes the linear endomorphism of $\p=\g/\kk$ induced by   $\Ad_G(k)\colon \g\to\g$, which is well defined because $\kk$ is a Lie subalgebra.

\begin{prop}\label{CALCULOEXPLICIT}\  
\begin{enumerate}
\item
The complex $\Omega_{G_\sharp}(N)$ of $G_\sharp$-invariant forms is isomorphic to the complex   $(\Lambda^r \p)^*_{K_\sharp}$ of alternate multilinear forms  $\p^r\to \R$ which are  $\Ad_\p(K_\sharp)$-invariant \cite[p. 458]{Greub1976};
\item
If the pair $(\g,K_\sharp)$ is reductive, then  $(\Lambda^r \p)^*_{K_\sharp}$  is isomorphic to the complex $L_{K_\sharp}(\Lambda^r\p,\R)$.
\end{enumerate}
\end{prop}

\begin{cor}
Let the pair $(\g,K_\sharp)$  be reductive. 
Then   $H_{G_\sharp}(N)$ is isomorphic to  $H(\g,K_\sharp;\R)$.
\end{cor}

Then, from Corollary \ref{dualpoincare}, Theorem \ref{hcero} and Proposition \ref{CALCULOEXPLICIT}, we are able to prove the following result. 

\begin{prop} \label{equivuni}
 If the pair $(\g,K_\sharp)$  is reductive and the groups $G_\sharp$ and $K_\sharp$ are strongly unimodular, then  $H^q_{G_\sharp}(N)\not =0$, where $q=\dim N$. In fact, $H^q_{G_\sharp}(N)=\R$. Conversely, when $G_\sharp$ is connected, the condition $H^q_{G_\sharp}(N)\not =0$ implies that $G_\sharp$ and $K_\sharp$ are strongly unimodular.
 \end{prop}
 
\subsection{Unimodular foliations}
We apply the results of the last paragraph to the transversely homogeneous foliation $\F$ on the manifold $M$.

\begin{defi}[\cite{Reinhart1959}] The differential form $\alpha$ on $M$ is {\em base-like} for the foliation $\F$ if 
it is invariant and horizontal, that is,  $i_X\alpha=0$ and $i_X\dd \alpha=0$ for any vector field $X$ tangent to the foliation. 
\end{defi}
We shall denote by $(\Omega^\bullet(M),\dd)$ the De Rham complex of differential forms on $M$, and by $\Omega^\bullet(M/\F)$ the subcomplex of base-like forms. The {\em base-like} or {\em basic} cohomology of the foliation $\F$  is the cohomology $H(M/\F)$ of this subcomplex.

\begin{defi}The foliation $\F$ is {\em unimodular} if $H^q(M/\F)\neq 0$, for $q=\dim N=\codim \F$.
\end{defi}

The following result is a direct consequence of the structure theorem \ref{desth}.
We shall need one previous Lemma:

\begin{lema} \label{sfc}
Let $f\colon \widetilde M\to N$ be a submersion with connected fibers and let   $\widetilde \F=f^* \pt$ be the simple foliation defined by $f$.   Then $H(\widetilde M/\widetilde \F)\cong H(N)$.
\end{lema}

The following Theorem was first proved by Blumenthal in \cite{Blumenthal1980} under some more restrictive hypothesis.

\begin{teo} \label{cohth} Let $\F$ be a $N$-transversely homogeneous foliation on the manifold $M$, 
with $N$ connected. If there is a developing map $f$ which is surjective and with connected fibers then the base-like cohomology $H(M/\F)$ is isomorphic to $H_{\Gamma}(N)$.
\end{teo}
\begin{proof}
Let $h \colon \Aut(p)\cong \Gamma\subset G_\sharp$ be the isomorphism given by the structure theorem \ref{desth}.
The covering map
$p$ induces an isomorphism $p^*\colon \Omega^\bullet(M/{\F})\to \Omega^\bullet_{\invar}(\widetilde M/\widetilde{\F})$ between the base-like forms for $(M,\F)$  and the base-like forms for  $(\widetilde M,\widetilde \F)$ which are invariant by the action of ${\rm Aut}(p)$.

Now it is enough to check that $f^*\colon H^r_{\Gamma}(N) \to  H^r_{\invar}(\widetilde M/\widetilde \F)$ is an isomorphism.
\end{proof}
Theorem \ref{MAININTRO} gives sufficient conditions for the foliation $\F$ to be unimodular. For a discussion on the surjectiveness and connectedness of the fibers of the developing map see  Remark \ref{CONECCTEDFIB} and Proposition \ref{friemannianasikcompacto}.

\begin{teo} \label{confibrasconexas}
Let $\F$ be an $N$-transversely homogeneous foliation   on the compact manifold $M$, which admits a developing map with connected fibers. We assume that $N=G_\sharp/K_\sharp$ is connected.
If  $G_\sharp$   is strongly unimodular, $K_\sharp$ is compact and strongly unimodular, and $\overline \Gamma$ is unimodular,  then the foliation $\F$ is unimodular.  
\end{teo}
\begin{proof}
Since $M$ and $K_\sharp$ are compact,  Proposition \ref{friemannianasikcompacto} states that the developing map $f$ is a (surjective) locally trivial bundle.  Since the fibers of $f$ are connected,  Theorem \ref{cohth} ensures that  $H(M/\F)\cong H_{\overline \Gamma}(N)$. On the other hand,  the pair $(\g,K_\sharp)$ is reductive, by Proposition \ref{EFFISOM}. Finally, since $G_\sharp$ and $K_\sharp$ are strongly unimodular, we know that $H^q_{G_\sharp}(N)\not =0$, $q=\codim\F$, by Proposition \ref{equivuni}.

Now, since $M$ and $K_\sharp$ are compact, Proposition \ref{fblucompacto2} says that $W=\overline\Gamma\backslash G_\sharp$ is compact, so we have the hypothesis   to apply Theorem \ref{MAININTRO} and to obtain 
 that
 $$H(M/\F)=H_\Gamma(N)=H_{\overline\Gamma}(N)\not =0.\qedhere$$
\end{proof}

\subsection{Example}
In this subsection we illustrate some of the results of the paper with an example.

Let us consider the transitive action of $G_0=\mathrm{SL}(2,\R)$ on the  complex upper half-plane $N={\mathbb H}$, given by
$$\begin{bmatrix}x&y\cr z&t\cr\end{bmatrix}\cdot \omega =\dfrac{x\omega+y}{z\omega+t}.$$
The isotropy of $\omega=i$ is the subgroup $K_0= \mathrm{SO}(2)$, which is compact and connected.  The normal core is the only proper normal subgroup of $G_0$, that is,  is $\{\pm I\}$, so  $G_{0\sharp}=\mathrm{PSL}(2,\R)$ and $K_{0\sharp}=\mathrm{SO}(2)/\{\pm I\}$. Let $\Gamma_0\subset G_{0\sharp}$ be a discrete cocompact subgroup.  We have a transversely homogeneous foliation on the compact manifold $M=\Gamma_{0\sharp}\backslash G_{0\sharp}$,
whose holonomy is $\Gamma_0$ and whose developing map
$f\colon G_{0\sharp} \to N$
is given by $f(A)=A\cdot i$.

If $K_{0\sharp}\cap \Gamma_0=\{ I\}$ then the manifold $M$ is the unitary tangent bundle over $\Gamma_0\backslash {\mathbb H}$ and the foliation is defined as a fiber bundle. If $K_{0\sharp}\cap \Gamma_0\not =\{ I\}$ then the leaves have holonomy and the foliation is defined as a bundle over a Satake manifold  
\cite[p. 89]{Molino1988}.

We now check the hypothesis of Theorem \ref{confibrasconexas}, in order to show that the foliation is unimodular.

The fibers of $f$ are connected. The isotropy is compact, so there is an invariant metric on $N$, namely, $\mathrm{PSL}(2,\R)$ is the  group of orientation preserving isometries of the hyperbolic metric $(dx^2+dy^2)/y^2$. The manifold $W=\Gamma_0\backslash G_{0\sharp}$ is compact.  The group $K_{0\sharp}$ is (strongly) unimodular, because it is compact. The subgroup $\Gamma_0$ is discrete, hence unimodular. Finally, the Lie group $\mathrm{PSL}(2,\R)$  is unimodular too, because its Lie algebra is unimodular: namely, it admits a basis $X,Y,Z$ subject to the relations  $[X,Y]=2Y$; $[X,Z]=-2Z$ and $[Y,Z]=X$.

This example can be easily generalized to $\SL(n,\R)$.

\section{Non-connected fibers}\label{NONCONN}
We study now the case when the fibers of the developing map are not connected. We start with a connected Lie group $G_0$ which acts transitively on $N$, but we do not assume  the action to be effective.  This will allow us to model the foliation on the universal covering $\widehat N$, where the developing map will have connected fibers. 

\subsection{Auxiliary constructions} \label{construcciondelextended}
Let $N=G_0/K_0$ be a homogeneous space, where $G_0$ is {\em connected} (we  do not assume  $G_0$ to be simply connected).  If $\pi\colon \widehat G_0\to G_0$ is the universal covering of $G_0$, we have
$$N=\widehat G_0/\pi^{-1}(K_0).$$
For the sake of simplicity, we shall denote $\widehat K_0=\pi^{-1}(K_0)$, even if this group may not be the universal covering of $K_0$.

We maintain our notations $\core(K_0)$ (respectively $\core(\widehat K_0)$) for the normal core of the action of $G_0$ (resp.  $\widehat G_0$) on $N$.

\begin{prop}\label{coredeluniversal} We have
\begin{align}
(\widehat G_0)_\sharp=\widehat G_0/\core(\widehat K_0) &\cong G_0/\core(K_0)=(G_0)_\sharp,\\
(\widehat K_0)_\sharp=\widehat K_0/\core(\widehat K_0) &\cong K_0/\core(K_0)=(K_0)_\sharp.
\end{align}
\end{prop}

\begin{prop}\label{cubierta} Let $(\widehat K_0)_e$ denote the connected component of the identity of the subgroup $\widehat K_0$. Then, the universal covering of  $N$ is $\widehat N=\widehat G_0/(\widehat K_0)_e$, and the fundamental group $\pi_1(N)$ is isomorphic to $\widehat K_0/(\widehat K_0)_e$. 
\end{prop}

In order to get an equivariant map onto the homogeneous space $\widehat N$, we need to enlarge the group $G_{0\sharp}=G_0/\core(K_0)$, which acts effectively on $\widehat N$, by the group of deck transformations of the covering $\pi_N\colon \widehat N \to N$. More precisely we have the following technical definition.

\begin{defi}\label{EXTENDED}
Let $G$ denote the Lie group $$G\eqdef\widehat G_0/\core((\widehat K_0)_e)\times \widehat K_0/(\widehat K_0)_e,$$
which we call the {\em extended} group.  
\end{defi}

This extended group acts transitively on $\widehat N=\widehat G_0/(\widehat K_0)_e$, where the action is given by
$$([g],[k])\cdot [h] =[ghk\inv],\quad g,h\in \widehat G_0, k\in \widehat K_0.$$

\allowdisplaybreaks

\begin{prop} \label{isotropiascoinciden}\ 
\begin{enumerate}
\item
The   isotropy  of the action at the point $[e]\in\widehat N$  is the subgroup
$$i(\widehat K_0)=\{ ([k],[k])\colon k\in \widehat K_0\},$$
which is isomorphic to
 $$K\eqdef\widehat K_0/\core((\widehat K_0)_e).$$
\item
The
normal core $\core_G(K)$ of this action is 
$$
i(\core(\widehat K_0))=\{ ([k] ,[k])\colon k\in \core(\widehat K_0)\},
$$
which is isomorphic to the  abelian group
$$\core(\widehat K_0)/\core((\widehat K_0)_e).$$
 \item
The Lie group $G_\sharp=G/\core(K)$ acts transitively and effectively on $\widehat N$, with isotropy  $$K_\sharp=K/\core(K)\cong \widehat K_{0\sharp}  \cong K_{0\sharp}.$$
\end{enumerate}
\end{prop}

\begin{remark} Notice that the Lie group $G_\sharp$ may not be connected. In fact, $\pi_0(G_\sharp)=\pi_0(K_\sharp)$, where $K_\sharp=K/\core(K)$, and the connected component of the identity of $G_\sharp$ is diffeomorphic to $ \widehat G_0/\core((\widehat K_0)_e)$.
\end{remark}

\begin{lema}\label{cubiertaauxiliar}
The projection
\begin{equation}\label{morfismoq}
q\colon \widehat G_0/\core((\widehat K_0)_e)\to G_{0\sharp}=\widehat G_0/\core(\widehat K_0)
\end{equation}
 is a covering of Lie groups, with automorphism group the abelian group $\core(\widehat K_0)/\core((\widehat K_0)_e)$.
\end{lema}

\begin{prop}\label{COVERCERO}
The Lie group  $G_\sharp$  is a  (maybe non-connected) covering of the connected Lie group  $G_{0\sharp}$. More precisely,  $G_\sharp$ is an extension of  $G_{0\sharp}$ by $\widehat K_0/(\widehat K_0)_e$.
\end{prop}

\begin{proof}
Let us denote by
$i(\core(\widehat K_0))$ the subgroup of the extended group $G$ (Definition \ref{EXTENDED}) given by
$$\{([k],[k])\in G\colon k\in \core(\widehat K_0)\}.$$
Then
$$G_\sharp= G/i(\core(\widehat K_0)).$$
Consider the morphism 
$$j\colon \widehat K_0/(\widehat K_0)_e\to G_\sharp,
$$
given by
$$
j([k])=[([e],[k])].
$$

This morphism is injective because $[([e],[k])]=[([e],[e])]$ would imply that $([e],[k])\in i(\core(\widehat K_0))$, hence $[k]=[e]\in \widehat K_0/(\widehat K_0)_e$. 

Now, the projection 
$$E\colon G_\sharp=G/i(\core(\widehat K_0)) \to G_{0\sharp}=\widehat G_0/\core(\widehat K_0)$$
will be defined as 
\begin{equation}\label{defcubiertadeE}
E([([g],[k])])=q([g]),
\end{equation}
where $q$ is the morphism \eqref{morfismoq}. 

The projection $E$ is well defined, because if $([gk^\prime],[kk^\prime])$, with $k^\prime\in \core(\widehat K_0)$, is another representative of the class  $[([g],[k])]$ in $G_\sharp$, then $q([gk^\prime])=q([g])$. Trivially, the map $E$ is surjective.

It remains to show that 
$$
\widehat K_0/(\widehat K_0)_e\stackrel{j}\rightarrow G_\sharp\stackrel{E}\rightarrow \widehat G_0/\core(\widehat K_0)
$$
is an exact sequence, that is, $\ker E=\im j$. 

First,
$$
[([g],[k])]\in \ker E\Leftrightarrow q([g])=[e] \Leftrightarrow g\in \core(\widehat K_0).
$$
Then, the class $[g] \in \widehat G_0/\core((\widehat K_0)_e)$ belongs to  $\core(\widehat K_0)/\core((\widehat K_0)_e)$ and
$$
[([g],[k])]=[([e],[kg^{-1}])\cdot ([g],[g])]=j([kg^{-1}]),
$$
because $([g],[g])\in i(\core(\widehat K_0))$. 

Conversely,
$$
E j([k])=E([([e],[k])])=q([e])=[e].\qedhere
$$
\end{proof}

\begin{cor} \label{unimodularidaddee}
The Lie group $G_\sharp$ is strongly unimodular if and only if $G_{0\sharp}$ is unimodular. 
\end{cor}

\begin{proof}
Immediate from Proposition  \ref{funcionmodulardee2}.
\end{proof}

\subsection{Unimodular foliations again}\label{unifolsection} Our main result of this paragraph is analogous to Theorem \ref{confibrasconexas}, but now we do not ask  the fibers of the developing map to be connected. In contrast, we need that the Lie group $G_0$ acting on $N$ be connected.

\subsubsection{Transverse model}
Let $\F$ be an $N$-transversely homogenous foliations on the {\em compact} manifold  $M$, with transverse model $N=G_0/K_0$, and holonomy group $\Gamma_0\subset G_{0\sharp}$. 

It was proved by Blumenthal in \cite[Theorem 4.1]{Blumenthal1979} that the universal covering $\widehat M$ of $M$ fibers over the universal covering of $\widehat N$,  the fibers being the leaves of the lifted foliation. Our next results refine this idea.

Let $\widetilde \F=p^*\F$  be the lifted foliation of $\F$ to the covering $\widetilde M$ given by the structure theorem \ref{desth}. Remember that $\widetilde \F$ is the simple foliation defined by the developing map
$f\colon \widetilde M\to N$. Let $\widehat N$ be the universal covering of $N$. This manifold $\widehat N$ is a  $G_\sharp$-homogeneous space, where $G_\sharp$  is the extension of $G_{0\sharp}$ given in Proposition \ref{COVERCERO}.

\begin{lemma} \label{ftildenfoliacion1}
The foliation $\widetilde \F$ on $\widetilde M$ is a transversely homogeneous foliation with transverse model  $\widehat N$. More precisely, if $\widehat M$ is the universal covering of $M$, the map $f$ lifts to a submersion $\widehat f\colon \widehat M\to \widehat N$, which  is a locally trivial bundle with connected fibers when $K_0$ is compact. 
 
Moreover, the holonomy  subgroup $\widetilde \Gamma_0$ of  $\widetilde \F$  is the image of the morphism 
$$\pi_1(f)\colon \pi_1(\widetilde M)\to\pi_1(N)\cong \widehat K_0/(\widehat K_0)_e\subset G_\sharp.$$
\end{lemma}
\begin{proof}
The existence of $\widehat f$ is granted by the homotopy lifting property of the covering  $\pi$, because $\widehat M$ and $\widehat N$ are simply connected.
Let us check that $\widehat f$ is equivariant for the morphism $f_*=\pi_1(f)$:

If $\gamma\in\Aut p=\pi_1(\widetilde M)$ and $\widehat x\in \widehat M$, denote $\widetilde x=p(\widehat x)\in \widetilde M$. The loop $\gamma$ with base point $\widetilde x$ lifts to a path $\widehat \gamma$ in $\widehat M$ with initial point $\widehat x$ and end point $\widehat\gamma\cdot \widehat x{\mathrel{:\mkern-0.25mu=}}\widehat \gamma(1)$. On the other hand, we have fixed base-points $x_0\in M$, $\widetilde x_0\in \widetilde M$ and $\widehat x_0$. For any path $\widehat\delta$ joining $\widehat x_0$ with $\widehat x$ we shall have the image path $\alpha= (f\circ p)(\widehat\delta)$ in $N$
joining  $(f\circ p)(\widehat x_0)$ con $(f\circ p)(\widehat x)$. By lifting this path to $\widehat N$ we shall have a path $\widehat \alpha$ with initial point $\widehat f (\widehat x_0)=\widehat n_0$ (a base-point previously fixed) and end point $\widehat f (\widehat x){\mathrel{:\mkern-0.25mu=}}\widehat \alpha(1)$.

Now we compute $\widehat f (\gamma\widehat x)$. We take the path $\widehat\delta \ast\widehat \gamma$ in $\widehat M$ , joining $\widehat x_0$ to $\gamma\widehat x$. Passing to $N$ through $f\circ p$ we obtain a path 
$$\beta=(f\circ p) (\widehat\delta\ast\widehat\gamma)=\alpha\ast f_*(\gamma),$$ which lifts to  $\widehat \beta= \widehat \alpha\ast \widehat{f_*(\gamma)}$. In this way, 
$$
\widehat f (\gamma \widehat x)=\widehat \beta(1)=\widehat{f_*(\gamma)}(1)=f_*(\gamma)\cdot \widehat f (\widehat x). 
$$

On the other hand, when $K_0$ is compact, an argument similar to that of Proposition \ref{friemannianasikcompacto} proves that $\widehat f$ is a locally trivial fiber bundle. The connectedness of the fibers follows from the homotopy long exact sequence.
\end{proof}

Consider the diagram
\begin{equation}\label{diagflevantada}
\xymatrix{
\widehat M \ar[r]^{\widehat f }\ar[d]_{\widetilde p}&\widehat N\ar[d]^\pi\\
\widetilde M \ar[r]^f\ar[d]_p&N\\
M}
\end{equation}
where $\widehat p=p\circ\widetilde p$ is the universal covering of $M$. Remember that the holonomy of $\F$ as an $N$-transversely homogeneous foliation is denoted by $\Gamma_0\subset G_{0\sharp}$; it is the image of a morphism  $h\colon \pi_1(M)\to G_{0\sharp}$ such  that $f$ is $h$-equivariant. We need to find a morphism $\widehat h\colon \pi_1(M)\to G_\sharp$ making  $\widehat f $  a $\widehat h$-equivariant map.

Notice that, for a given $\gamma\in \pi_1(M)=\Aut(p \circ \widetilde p )$ and $\widehat x\in \widehat M$,  we have (we denote $\widetilde x=\widetilde p(\widehat x)$,  $x=\widehat p(\widehat x)=p(\widetilde x)$ and $\overline\gamma\in \Aut(p)$)
\begin{align}
\label{pibarfgamma}\pi(\widehat f (\gamma\widehat x))&=f(\widetilde p(\gamma \widehat x))\\
\nonumber&=f(\overline\gamma\widetilde x)\\
\nonumber&=h(\gamma)\cdot f(\widetilde x)\\
\nonumber&=h(\gamma)\cdot f(\widetilde p(\widehat x))\\
\nonumber&=h(\gamma)\cdot \pi(\widehat f (\widehat x)).
\end{align}

But $h(\gamma)\in G_{0\sharp}=\widehat G_{0\sharp}$ does not act directly on $\widehat N$, so we shall use an arbitrary global section  $s$   of the covering $q$ given in \eqref{morfismoq}. We can assume that $s([e] )=[e] $. The section $s$ may not be a group morphism, so we define
$$
c\colon  \widehat G_{0\sharp} \times  \widehat G_{0\sharp} \to  \core(\widehat K_0)/\core((\widehat K_0)_e)$$
as
\begin{equation} \label{cocicloextension}
c([g]  ,[g^\prime]  )=s([g] ) \cdot s([g^\prime] )\cdot s ([gg^\prime]) ^{-1},
\end{equation}
which satisfies the usual cocycle condition.

Remember from Proposition \ref{coredeluniversal} that $G_{0\sharp}=\widehat G_0/
\core(\widehat K_0)$. We represent the class of $g\in \widehat G_0$ by $[g]$, while we shall use the notation $[g]_\sharp$ for the class of $g$ in the total space $\widehat G_0/
\core((\widehat K_0)_e)$ of the covering $q$ in \eqref{morfismoq}. So, this element $[g]_\sharp$   acts on $\widehat N=\widehat{G_0}/(\widehat{K_0})_e$.

\begin{lema} \label{pidegnesgpiden}
For $\widehat n\in \widehat N$ and $[g] \in G_{0\sharp}$ we have
$$
[g]  \cdot \pi(\widehat n)=\pi(s([g] )\cdot \widehat n).
$$
\end{lema}
\begin{proof}
Consider the commutative diagram
$$
\xymatrix{
\widehat N \ar[r]^{\widehat \lambda([g]_\sharp )}\ar[d]_\pi&\widehat N\ar[d]^\pi\\
N\ar[r]_{\lambda([g] )}&N}
$$
where $[g]=q([g]_\sharp)$.
It follows for $[g]_\sharp=s([g])$ that
\begin{align*}
\pi(s([g])\cdot \widehat n)&=(\pi\circ \widehat \lambda (s([g]))(\widehat n)\\
&=(\lambda([g])\circ \pi)(\widehat n)\\
&=[g]\cdot \pi(\widehat n). \qedhere
\end{align*}
\end{proof}

As a consequence, in \eqref{pibarfgamma} we shall have
$$
\pi(\widehat f (\gamma\widehat x))=\pi(s(h(\gamma))\cdot \widehat f (\widehat x)).
$$
That means that there exists
 $\xi(\gamma,\widehat x)\in \Aut(\pi)\cong \widehat K_0/(\widehat K_0)_e$  such that
\begin{equation}\label{equivflevantada}
\widehat f (\gamma\widehat x)=\xi(\gamma,\widehat x)\cdot s(h(\gamma))\cdot \widehat f (\widehat x).
\end{equation}
\begin{lemma}$\xi$ only depends on $\gamma$.
\end{lemma}
\begin{proof}
Since $\Aut(\pi)\cong \widehat K_0/(\widehat K_0)_e$ is a discrete group, it is enough to prove that the map  $\xi(\gamma,-)\colon \widehat M\to  \widehat K_0/(\widehat K_0)_e$ is continuous, because the manifold $\widehat M$ is connected. But it is not hard to prove that $\xi(\gamma,-)$ is locally constant, because $\widehat f\colon \widehat M \to \widehat N$ maps trivializing open sets of the covering $\widetilde p$ into trivializing open coverings of $\pi$.\end{proof}

So we have a map $\xi\colon \pi_1(M)\to  \widehat K_0/(\widehat K_0)_e$.  But this map is not a group morphism, because, for given $\gamma_1,\gamma_2\in \pi_1(M)$, we have
$$
\xi(\gamma_1\gamma_2)=\xi(\gamma_1)\cdot \xi(\gamma_2)\cdot c_{12},
$$
where $c_{12}=c(h(\gamma_1),h(\gamma_2))\in \core(\widehat K_0)/\core((\widehat K_0)_e)$, as in \eqref{cocicloextension}.

However, 
the map  $\widehat h\colon \pi_1(M)\to G_\sharp$ given by 
\begin{equation}\label{defhath}
\widehat h(\gamma)=[(s(h(\gamma)),\xi(\gamma))]\in G_\sharp.
\end{equation}
is a group morphism, as it is straightforward to check.

Moreover, the submersion $\widehat f $ is   $\widehat h$-equivariant. This gives the foliation $\F$ on $M$ a structure of   $\widehat N$-transversely homogeneous foliation.%

\begin{prop}  \label{ftildenfoliacion2}
The foliation $\F$ has a structure of  $\widehat N$-transversely homogeneous foliation, when  $\widehat N$  is considered as a $G_\sharp$-homogeneous space.
\end{prop}

\subsubsection{Holonomy groups}

We shall denote by $\Gamma\subset G_\sharp$ the holonomy group of $\F$ when it is considered as a $\widehat N$-transversely homogeneous foliation. Remember that $\Gamma_0\subset G_{0\sharp}$ is the holonomy of the $N$-transversely homogeneous foliation $\F$.

\begin{lemma}\label{Gammaescubiertadegamma0} Let  $E\colon G_\sharp \to G_{0\sharp}$ be the projection given in Proposition \ref{COVERCERO}. Then the image of $\Gamma$ is $\Gamma_0$, that is,
 $E(\Gamma)=\Gamma_0$.
\end{lemma}
\begin{proof}
Since $\Gamma_0=\im h$, with $h\colon \pi_1(M)\to G_{0\sharp}$, the result follows from equations \eqref{defcubiertadeE} and \eqref{defhath}, because for a given $\gamma\in \pi_1(M)$ we have
$$
E([(s(h(\gamma)),\xi(\gamma))])=q(s(h(\gamma)))=h(\gamma). \qedhere
$$
\end{proof}

We need two Lemmas, previous to the next important Proposition \ref{unimodularidaddeGamma}.

\begin{lema} \label{eldeenmediocerrado}
Let $A\subset B\subset C$ three subgroups of a Lie group $G$, such that $A,C$ are closed in $G$, and the set  $C/A$ is finite. Then $B$ is closed in $G$.
\end{lema}
\begin{proof}
We choose representatives $c_1,\dots,c_N\in C$ of the cosets in $C/A$. 
 
Let $\{x_n\}_{n\in \N}$ be a sequence in  $B$ converging to some  $x\in G$. Since each $x_n$ belongs to some coset, there must be some $c\in\{c_1,\dots,c_N\}$ which appears an infinite number of times. Then there is a convergent subsequence $\{x_{m}\}$, with $[x_m]=[c]$, so $x_m=c\cdot y_m$  for some $y_m\in A$. Notice that  $c=x_m\cdot y_m^{-1}\in B$. Then we have
$$
x=\lim_{m\to\infty}c\cdot y_{m}=c\cdot\lim_{m\to\infty}y_{m}.
$$
Since  $A$ is closed in $G$, we have   $\lim_{m\to\infty}y_{m}\in A\subset B$, hence $x\in B$.
\end{proof}

\begin{lema} \label{clusorecociente}
Let $G$ be a Lie group, and let   $B\subset A$ two  subgroups of $G$ such that the set $A/B$ is finite. Then the space $\overline A/\overline B$ is finite too.
\end{lema}

\begin{proof}
We shall prove that there is a finite set $a_1,\dots,a_n \in \overline A$ such that each $x\in \overline A$ belongs to some $a_i \overline B$.

In fact, we shall take representatives $a_1,\dots,a_N$ of each  coset $A/B$. Then, if $x=\lim_{n\to \infty} x_n$, with $x_n\in A$, since each $[x_n]$ determines a coset, there must be some $a\in\{a_1,\dots, a_N\}$ which appears an infinite number of times. That means that there is a subsequence $\{x_m\}$ converging to $x$ such that $[x_m]=[a]$ for all $m$, that is,
$x_m=a\cdot b_m$, with $b_m\in B$.

From
$$
x=\lim_{m\to\infty}a\cdot b_{m}=a\cdot\lim_{m\to\infty}b_{m},
$$
it follows that $a^{-1}x\in\overline B$, hence   $x\in a\cdot\overline B$.  \end{proof}

\begin{prop} \label{unimodularidaddeGamma}
Assume that $K_{0\sharp}$ is compact and that   fibers of the developing map $f\colon \widetilde M \to N$ have a finite number of connected components. Then:
\begin{enumerate}
\item
The image of the closure of \ $\Gamma$ in $G_\sharp$ is the closure of \ $\Gamma_0$ in $G_{0\sharp}$, that is,
$E(\overline \Gamma)=\overline \Gamma_0$. Analogously,   $E((\overline\Gamma)_e)=(\overline\Gamma_0)_e$.
\item
$\overline \Gamma$ is unimodular if and only if $\overline \Gamma_0$ is unimodular. 
Analogously,  $(\overline \Gamma)_e$ is unimodular
 if and only if $(\overline \Gamma_0)_e$ is unimodular.
\end{enumerate}
\end{prop}
\begin{proof}
\begin{enumerate}
\item \label{imagendeoverlinegamma}
Since $K_{0\sharp}$ is compact, we know from Proposition \ref{friemannianasikcompacto} that the developing map  $f\colon \widetilde M\to N$ is a fibration. Denote by $F$ its generic fiber, and let $\widetilde \Gamma_0$ be the image of the holonomy morphism $\widetilde h =\pi_1(f)$ given in Lemma \ref{ftildenfoliacion1}. From the homotopy  long exact sequence we have
\begin{equation}\label{ecuaciondepi0def}
 \pi_0(F)\equiv (\widehat K_0/(\widehat K_0)_e)/\widetilde\Gamma_0.
\end{equation}
Consider the covering  $E\colon G_\sharp\to G_{0\sharp}$, given by 
$$
E([([g],[k])])=q([g]).
$$
as in \eqref{defcubiertadeE}.
We know from Lemma \ref{Gammaescubiertadegamma0} that $E(\Gamma)=\Gamma_0$, so  
$$
 E^{-1}(\Gamma_0)=\Gamma\cdot \ker  E.
$$

Since the covering $E$ restricts to a morphism $\Gamma\to \Gamma_0$, with kernel $\widetilde\Gamma_0$, we have
\begin{equation}\label{pimenosunodeGammaentreSigma}
\ker  E/\widetilde\Gamma_0\cong  E^{-1}(\Gamma_0)/\Gamma.
\end{equation}
Hence, combining Equations  \eqref{ecuaciondepi0def} and  \eqref{pimenosunodeGammaentreSigma} we have that 
$$
\pi_0(F) \cong  E^{-1}(\Gamma_0)/\Gamma
$$
is a finite set. It follows from  Lemma \ref{clusorecociente} that
$$
\overline{ E^{-1}(\Gamma_0)}/\overline \Gamma\cong  E^{-1}(\overline\Gamma_0)/\overline \Gamma
$$
is finite too, and this implies that 
\begin{equation}\label{paracomprobar}
 E(\overline \Gamma)=\overline\Gamma_0,
\end{equation}
as we shall check in the next paragraph. By dimension reasons, this will imply that
$ E((\overline\Gamma)_e)=(\overline\Gamma_0)_e$. 

So, let us check \eqref{paracomprobar}. Let $H= E^{-1}( E(\overline\Gamma))$ the saturated of   $\overline \Gamma$. We have   $\overline \Gamma\subset H\subset  E^{-1}(\overline \Gamma_0)$, with $ E^{-1}(\overline \Gamma_0)/\overline \Gamma$  finite, so Lemma \ref{eldeenmediocerrado}, states that $H$ is a closed subgroup of $G_\sharp$, which means that $ E(\overline \Gamma)$  is a closed subgroup of $G_{0\sharp}$.

\item
It is immediate from part (\ref{imagendeoverlinegamma}) and Proposition \ref{funcionmodulardee2}.
\end{enumerate}
\end{proof}

This will allow us to generalize Theorem \ref{confibrasconexas} to foliations such that the fibers of the developing map are not connected, but have a finite number of components.

\begin{teo} \label{tercerteoremadeuni}
Let $\F$ be a $N$-transversely homogeneous foliation on the compact manifold $M$, where $N=G_0/K_0$. Assume that the Lie group $G_0$ is connected and that the fibers of the developing map have a finite number of connected components.  Assume moreover that the Lie group $(K_0)_\sharp$ is compact. 
If the Lie groups $G_{0\sharp}$ and $\overline \Gamma_0$ are unimodular, and $(K_0)_\sharp$ is strongly unimodular,
then the foliation $\F$ is unimodular.
\end{teo}
\begin{proof}
We take the universal covering $\pi\colon \widehat G_0\to G_0$ and $\widehat K_0=\pi^{-1}(K_0)$. By Proposition \ref{coredeluniversal} we know that 
$$(\widehat G_0)_\sharp=\widehat G_0/\core(\widehat K_0) \cong G_0/\core(K_0) =G_{0\sharp}$$ and that 
$$(\widehat K_0)_\sharp=\widehat K_0/\core(\widehat K_0) \cong K_0/\core(K_0)=(K_0)_\sharp.$$
Since $(K_0)_\sharp$ is compact, the developing map $f\colon \widetilde M\to N$, as well as its lifting $\widehat f\colon  \widehat M \to \widehat N$ to the universal covering are locally trivial bundles. By Proposition   \ref{ftildenfoliacion2} we can consider that the foliation $\F$ on $M$ models on $\widehat N=G_\sharp/K_\sharp$, where $G$ is the extended group given in section \ref{construcciondelextended}, and the isotropy $K$ verifies that  $K_\sharp\cong (K_0)_\sharp$ by Proposition \ref{isotropiascoinciden}.  The holonomy group of the latter foliation was denoted by $\Gamma\subset G_\sharp$. Moreover, the developing map $\widehat f\colon \widehat M\to \widehat N$ has connected fibers, so we can apply Theorem   \ref{confibrasconexas}, because:
\begin{enumerate}
\item
The Lie group $K_\sharp\cong (K_0)_\sharp$ is compact, and strongly unimodular, by hypothesis;
\item
The  Lie groups $G_\sharp$ and $\overline \Gamma$ are unimodular. The first one, by 
Proposition \ref{funcionmodulardee2}, because
$G_{0\sharp}$  is unimodular, by hypothesis.  On the other hand, since $\overline \Gamma_0$ is unimodular it follows that $\overline \Gamma$ is unimodular,
by  Proposition \ref{unimodularidaddeGamma}.
\end{enumerate}

Hence,  Theorem \ref{confibrasconexas}, ensures that the foliation $\F$ is unimodular.
\end{proof}

\section{Non-unimodular foliations}\label{NONUNIM}
 In \cite{Blumenthal1979},
Blumenthal studied the closures of the leaves of a transversely homogeneous foliation  on a compact manifold, assuming that the transverse group acts effectively on $N$ and that the isotropy group is compact. This allowed him to prove that the foliation induced on each closure is a transversely homogeneous foliation, this generalizing the corresponding Molino's 
\cite{Molino1988} 
result for Lie foliations. In this setting, the holonomy group of the induced foliation is contained in the group $\mathrm{Iso}(\widehat N)$, the complete group of isometries of the universal covering $\widehat N$ of $N$, endowed with an invariant metric.  

The advantage of our construction in Section \ref{construcciondelextended} is that it allows  to give an explicit definition of $\widehat N$, without excluding the non-effective case, and to prove that the holonomy group is contained in a much smaller group, namely the extended group given in Definition \ref{EXTENDED}, which can be computed explicitly.

This will allow us to prove the Theorem \ref{generalizaciondelrendus} that we stated in the Introduction, which is our main result in the second part of the paper, and that  generalizes an analogous result that we proved for Lie foliations in  \cite{Macias2005}.

 \subsection{The closure of the leaves}

We continue to study the $N$-transversely homogeneous foliation $\F$ on the compact manifold $M$, where $N=G_0/K_0$. We assume that $G_0$ is connected and that the group $(K_0)_\sharp=K_0/\core(K_0)$ is compact.  From Proposition \ref{friemannianasikcompacto} we know that there exists a $G_{0\sharp}$-invariant metric on $\widehat N$, 
and that $\F$ is a Riemannian foliation.
Thanks to  Proposition \ref{ftildenfoliacion2} we can consider $\F$ as a $\widehat N$- transversely homogeneous foliation, where $\widehat N$ is effectively acted by the Lie group $G_\sharp$ given in Proposition \ref{isotropiascoinciden}. The isotropy of this action is   $K_{0\sharp}$.

 We shall denote by $\Gamma n\subset \widehat N$ the orbit of the point $n\in \widehat N$ by the action of $\Gamma\subset G_\sharp$ .

\begin{lemma}\cite[Lemma 4.3]{Blumenthal1979}  The closure $\overline{\Gamma n}\subset \widehat N$ of the orbit equals $\overline \Gamma n$, the orbit of $n$ by the action of the closure $\overline\Gamma$ of $\Gamma$. 
\end{lemma}  

Hence, $\overline{\Gamma n}$  is a homogeneous space given by the transitive and effective action of the Lie group  $\overline \Gamma\subset G_\sharp$.

\begin{remark}
Notice that Blumenthal considers the closure  of $\Gamma$ inside the Lie group of isometries $\mathrm{Iso}(\widehat N)$, but the compactness of $K_{0\sharp}$  ensures that it equals the closure inside   $G_\sharp$, thanks to the following general result:
``Let $N=G/K$ be a homogeneous space with $K$ compact. Then $G_\sharp$ maps injectivelly into  $\mathrm{Iso}(N)$, as a closed subgroup.'' The proof is easy by using that the projection $g\in G \mapsto \lambda(g)(o)\in G/K$ is a proper map.
\end{remark}

\begin{prop}\label{foladhconexa}\cite[Theorem 4.4.]{Blumenthal1979}
 The foliation induced by  $\F$ on the closure $\overline L$ of the leaf $L$ is a transversely homogeneous foliation modeled by the manifold  $\widehat N_L=\overline\Gamma n$, where $n$ is the image by $\widehat f$ of any leaf (fiber) of $\widehat f$ projecting onto $L$.   \end{prop}
 
 Blumenthal's proof includes the formula
 $$p^{-1}(\overline L)=(\widehat f)^{-1}(\overline{\Gamma\, n}),$$ so this set is a saturated subset of $\widehat M$ for the fibration $\widehat f$. From  the structure Theorem \ref{desth} we have a diagram
\begin{equation} \label{diagramaadherencias}
\xymatrix{
p^{-1}(\overline L)\ar[r]^{ \widehat f}\ar[d]_{p'}&\widehat N_L=\overline \Gamma\,  n\\
\overline L}
\end{equation}

Since all along the paper we have asked the transverse homogeneous model to be connected, we need to refine the latter Proposition.

\begin{lema}
The connected component $(\widehat N_L)_n$ of  $\widehat N_L$ containing the point $n\in \widehat N$ is diffeomorphic to the quotient of
$(\overline \Gamma)_e$ by some compact subgroup.
\end{lema}

The following general result:
\begin{prop} \label{accioncomponenteconexa}
If  a Lie group $G$ acts transitively on a manifold $N$, with isotropy $K=G_p$ the isotropy at the point $p\in N$, then the connected component $G_e$ of the identity acts transitively on the connected component $N_p$ of $p\in N$, with isotropy   $G_e\cap K$.
\end{prop}

\begin{cor} \label{foladhconexa2}
The foliation induced by $\F$ on the closure $\overline L$ of any leaf $L$ is a $(\widehat N_L)_n$- transversely homogeneous foliation,  where  an intermediate closed Lie subgroup $(\overline \Gamma)_e\subset \overline\Sigma \subset \overline\Gamma$ acts transitively and effectively on $(\widehat N_L)_n$, with compact holonomy. Moreover, the developing map of this foliation has connected fibers. 
\end{cor}

Before proving this result we need an elementary Lemma.

\begin{lemma}\label{COVERINGSWELL} Let $p\colon \widehat M \to M$ be the universal covering of the manifold $M$, and let $P$ be a path-connected component of $p^{-1}(\overline L)$. 
Then, the restriction $p''\colon P \to \overline L$ of $p$ is a covering, whose automorphism group $\Aut (p'')$ is formed by the deck transformations
$ \gamma\in \Aut(p)$ such that   $\gamma(P)=P$.
\end{lemma}

\begin{proof}[Proof of Corollary \ref{foladhconexa2}]

First, we have that $ \widehat f (P)$ equals $(\overline \Gamma n)_{ n}$, the connected component of $\overline \Gamma  n$ containing $n$. This follows from the fact that, $f^\prime$ is a surjective open map, and that the fibers of $\widehat f$ are connected.

Taking into account Proposition \ref{accioncomponenteconexa}, we have the following diagram
\begin{equation} \label{diagadherenciasconexas}
\xymatrix{
P\ar[r]^{\widehat f}\ar[d]_{p''}&(\overline \Gamma_e)\,n\\
\overline L}
\end{equation}   

This will endow $\F\mid_{\overline L}$ with a structure of $(\overline\Gamma)_e n$-transversely homogeneous foliation,
if we are able to prove that $\widehat f$ is equivariant for some morphism  $h''$ defined on  $\Aut(p'')$.  Consider the group $\Sigma$ of the elements $g\in \Gamma$ such that the action of $g$ on $\overline\Gamma\,n\subset \widehat N$ sends the component $(\overline\Gamma)_e\,n$ onto itself and is the identity on the other components. Then the Lie group $\overline\Sigma$ acts transitively and effectively on the manifold $(\overline\Gamma)_e\,n$. Moreover, since $\widehat f$ is $h$-equivariant, from Lemma \ref{COVERINGSWELL} it follows that the restriction of $\widehat f$ to $P$ is equivariant for the restriction of $\widehat h$ to $h''\colon \Aut(p'') \to \overline\Sigma$.
\end{proof}

\subsection{Proof of the main result}
 In this section we shall prove Theorem \ref{generalizaciondelrendus}. The proof will be a consequence of our previous study of the structure of the foliation and the cohomological results we stated in section \ref{SECTCOHOMOLOGY}, plus the following classic result.

\begin{teo} [Tischler Theorem \cite{Tischler1970}]\label{TISCHLERTH}
Let $M$ be a compact differentiable manifold admitting a non-singular closed $1$-form. Then  $M$ fibers over $S^1$.
\end{teo}

The latter result can be reformulated in terms of Lie foliations, by considering the codimension one foliation  defined by the condition $\omega=0$.

\begin{cor}\label{TISCHLERCOR}
Let $M$ be a compact differentiable manifold endowed with a Lie foliation modeled by the abelian Lie group $\R$. Then  $M$ fibers over $S^1$.
\end{cor}

We divide the proof of Theorem \ref{generalizaciondelrendus} in several separate Propositions.

First, we know, from Propositions  \ref{isotropiascoinciden} and \ref{ftildenfoliacion2}, that $\F$ can be considered as a transversely homogeneous foliation modeled by $\widehat N=G_\sharp/K_\sharp$, where  $K_\sharp\cong K_{0\sharp}$. The holonomy group was denoted by $\Gamma\subset G_\sharp$. Remember that $\Gamma_0\subset G_{0\sharp}$ is the holonomy group of $\F$ when seen as an $N$-transversely homogeneous foliation.

Also, Theorem \ref{tercerteoremadeuni} guarantees that, since $\F$ is not unimodular by hypothesis, then either $G_{0\sharp}$  or $\overbar \Gamma_0$ is not unimodular. Depending on this there are different fibrations to consider.

{\em Step 1}. We begin by assuming that  $G_{0\sharp}$ is not unimodular. 

\begin{prop}\label{PARTEUNO}
If $G_{0\sharp}$ is not unimodular then $M$ fibers over $S^1$.
\end{prop}

\begin{proof}
We consider the modular function $m_0=m_{G_{0\sharp}}\colon G_{0\sharp}\to (\R^+,\cdot)$, as given in Definition \ref{seccionfuncmodular}.

Since $G_{0\sharp}$ is connected, and it is not unimodular by hypothesis, the morphism   $m_0$ is surjective. Moreover, since $K_{0\sharp}$ is compact its image  $m_0(K_{0\sharp})$ is trivial. 

So  $m_0$ passes to the quotient, and we can define a map
$$
m_N\colon N\to \R^+, \quad
m_N([g])=m_0(g).
$$
Take
$$\overbar f=\log m_N\circ f\colon \widetilde M \to \R$$ and $$\overbar h=\log m_0\circ h\colon \pi_1(M) \to \R,$$
where $f$ and $h$ are respectively the developing map and the holonomy morphism of the foliation $\F$.

The maps $\overbar f$ and $\overbar h$ give then the developing map and the holonomy morphism of a Lie foliation on $M$ (Example \ref{liefolprimeraparte}), once we have tested  the equivariance in Lemma \ref{FALTABA}. By applying Tischler's theorem \ref{TISCHLERCOR}, this will prove that $M$ fibers over $S^1$.

This ends the proof of Proposition \ref{PARTEUNO}.
\end{proof}

\begin{lema}\label{FALTABA}$\overbar f$ is $\overbar h$-equivariant.
\end{lema}

\begin{proof}
First, we prove that, for any $\gamma \in \Gamma_0$ and $[g]\in N$ we have
$$
m_N(\gamma \cdot [g])=m_0(\gamma)\cdot m_N([g]).
$$
In fact,
\begin{align*}
m_N (\gamma\cdot [g])&=m_N ([\gamma g])\\
&=m_0(\gamma g)\\
&=m_0(\gamma)\cdot m_0(g)\\
&=m_0(\gamma)\cdot m_N ([g]). 
\end{align*}

So, for given $x\in \widetilde M$ and $\gamma \in \pi_1(M)$ we shall have
\begin{align*}
\overbar f(\gamma x)&=\log m_N (f(\gamma x))\\
&=\log m_N (h(\gamma)\cdot f(x))\\
&=\log (m_0(h(\gamma))\cdot m_N (f(x))\\
&=\log m_0(h(\gamma))+\log m_N (f(x))\\
&=\overbar h(\gamma)+\overbar f(x),
\end{align*}
which proves the equivariance. 
\end{proof}

{\em Step 2.} We now assume that $\overline \Gamma_0$ is not unimodular.  However,  the connected component $(\overline \Gamma_0)_e$ may be or may not be unimodular. 

\begin{prop}If \ $\overline \Gamma_0$ and  $(\overline \Gamma_0)_e$ are not unimodular, then the closure $\overline L$ of any leaf $L$ fibers over $S^1$.
\end{prop}
\begin{proof}
Notice that $\dim \overline \Gamma_0\ge 1$, hence the modular function of Definition \ref{seccionfuncmodular} is defined.
Moreover, from Proposition \ref{unimodularidaddeGamma} it follows that $\overline\Gamma$ and  $(\overline \Gamma)_e$ are not unimodular. 

Now, Theorem \ref{foladhconexa2} ensures that the foliation induced by $\F$ on $\overline L$ is modeled by $(\widehat N_L)_n=(\overline\Gamma)_e\, n=\overline \Sigma/K_L$, where the isotropy $K_L$ is compact.
Analogously to the proof of Proposition \ref{PARTEUNO},   the modular function  
$$m \colon \overline \Sigma\to (\R^+,.)$$  
 passes to the quotient and we can define a map 
$$
\overbar m \colon (\widehat N_L)_n\to \R^+.
$$             
By considering the composition of $\log \overbar m$ with the developing submersion of the  foliation on $\overline L$, as well as  the composition of $\log m $ with   the holonomy morphism $\Aut(p'') \to \overline \Sigma$, we shall obtain an $\R$-Lie foliation on $\overline L$ and, again, by applying  Tischler's theorem, we shall arrive to the desired result, namely, that $\overline L$ fibers over $S^1$.
\end{proof}

It only remains to test the final and more difficult case.

\begin{prop}\label{TERCER}
If \ $\overline \Gamma_0$ is not unimodular, but $(\overline \Gamma_0)_e$ is unimodular, then the total space of the Blumenthal bundle $ \Gamma\backslash f^*(G_{0\sharp})\to M$ fibers over $S^1$.
\end{prop}

Before proving this Proposition we need several previous Lemmas.

From Proposition \ref{unimodularidaddeGamma}, we know that the group $(\overline \Gamma)_e$ is unimodular but $\overline\Gamma$ is not. 
We shall consider the universal covering $\pi_0\colon \widehat{G_{0\sharp}}\to G_{0\sharp}$. Let $$H=\pi_0^{-1}(\overline \Gamma_0)\subset \widehat{G_{0\sharp}}$$ be the inverse image of the closure $\overline\Gamma_0$.
By Proposition \ref{funcionmodulardee2} 
we know that $H$ is not unimodular. 

\begin{lema}\label{COMPUNIM}The connected component $H_e$ is unimodular.
\end{lema}

\begin{proof}
By Proposition \ref{funcionmodulardee2}  again,
we know that $H_0=\pi_0^{-1}((\overline\Gamma_0)_e)$  is unimodular, hence, by Proposition \ref{geunimodular}, the group $(H_0)_e$ is unimodular. In fact, we shall prove that this latter group equals $H_e$.

Obviously, $H_0\subset  H$, so  $(H_0)_e\subset H_e$.
On the other hand, $\pi_0(H_e)\subset \pi_0(H)=\overline\Gamma_0$, hence  $\pi_0(H_e)\subset (\overline\Gamma_0)_e$, by connectedness. It follows that   $H_e\subset H_0$ and by connectedness, $H_e\subset  (H_0)_e$.
\end{proof}

The following result is the crucial one. Let $m_H$ be the modular function of $H$.

\begin{lemma}\label{CRUCIAL}
It is possible to extend the non-trivial morphism of groups  $m_{H} \colon H \to (\R^+,\cdot)$  to a map $m\colon \widehat{G_{0\sharp}}\to {\R}^+$ such that:
\begin{enumerate}
\item
$m\vert_ H = m_H$,
\item
$m(h  y)=m(h)m(y)$ for all $h\in H$,  $y\in \widehat{G_{0\sharp}}$.
\end{enumerate}
\end{lemma}

\begin{proof}
Since $H_e$ and $H$ have the same Lie algebra, it is clear that the modular function of $H_e$ is the restriction of $m_H$ to $H_e$.  But $H_e$ is unimodular (Lemma \ref{COMPUNIM}), then
 $m_H(\gamma)=m_{H_e}(\gamma)=1$ for all $\gamma \in  H_e$. 
Hence there is a well-defined morphism  
\begin{equation}
\label{WELLDEF}
\overline m_{H}\colon H/H_e \to ({\R}^+,\cdot)
\end{equation}
 given by $\overline m_H([\gamma])=m_H(\gamma)$.

From Proposition \ref{fblucompacto} we know that the manifold 
$$W=H\backslash \widehat{G_{0\sharp}}=\overline\Gamma_0\backslash G_{0\sharp}$$ is compact. Since the group $\widehat{G_{0\sharp}}$ is simply connected, the universal covering of $W$ is the manifold $\widehat W= H\backslash \widehat{G_{0\sharp}}$, and the fundamental group of $W$ is $\pi_1(W)=H_e\backslash H$.
By applying logarithms, we have a group morphism
$$
\log  \overline m_{H}\colon  H/H_e=\pi_1(W)\to \R,
$$
so we can identify $\log  \overline m_H\in {\rm Hom}(\pi_1(W),\R)$ with a cohomology class $[\omega]\in H^1_{DR}(W)$ such that
\begin{equation}\label{integral}
\log \overline m_H([\alpha])=\int_\alpha \omega, \quad \mathrm{\ for\ all\ }[\alpha]\in \pi_1(W),
\end{equation}
where $[\alpha]$ denotes the homotopy class of the loop  $\alpha$ in $W$ with base point $[e]$.

Now, let $\pi\colon \widehat{G_{0\sharp}}\to W=H\backslash \widehat{G_{0\sharp}}$ be the natural projection. The $1$-form $\pi^*\omega$ in $\widehat{G_{0\sharp}}$ is closed, because $\omega$ is closed in $W$. Since $\widehat{G_{0\sharp}}$ is simply connected, hence $H^1(\widehat{G_{0\sharp}})=0$, the form $\pi^*\omega$ is exact, that is, there exists a map $f\colon \widehat{G_{0\sharp}}\to \R$ such that $\dd  f=\pi^*\omega$.
Since the translations by a constant do no affect the differential we can consider that $f(e)=0$.

With this condition, we have that $f$ verifies the following properties, whose proof will be delayed to Lemma \ref{VERIF}:

\begin{enumerate}
\item
$f(\gamma x)=f(\gamma)+f(x)$, for all $\gamma\in H$, $x\in \widehat{G_{0\sharp}}$;
\item 
$f_{\mid H}=\log m_{H}$.
\end{enumerate}

Let us take $m=e^f\colon \widehat{G_{0\sharp}} \to \R^+$. This is the map we were looking for, because
$$
m(\gamma)=e^{f(\gamma)}=e^{\log m_H(\gamma)}=m_H(\gamma),$$
for all $\gamma\in H$, and
$$m(\gamma y)=e^{f(\gamma y)}=e^{f(\gamma)+f(y)}=e^{f(\gamma)} e^{f(y)}=m(\gamma)  m(y),$$
for all $\gamma\in H, y\in \widehat{G_{0\sharp}}$.
\end{proof}

We now prove the Lemma announced a few lines above.
\begin{lemma}\label{VERIF}We have:
\begin{enumerate}
\item
$f(\gamma x)=f(\gamma)+f(x)$, for all $\gamma\in H$, $x\in \widehat{G_{0\sharp}}$;
\item 
$f_{\mid H}=\log \, m_{H}$.
\end{enumerate}
\end{lemma}
\begin{proof}\ 
{(1)}
Since $H\subset \widehat{G_{0\sharp}}$, we consider the composition  $f\circ L_\gamma\colon\widehat{G_{0\sharp}}\to \R$, where $L_\gamma$ denotes the left translation $L_\gamma(x)=\gamma x$. For the projection $\pi\colon\widehat{G_{0\sharp}}\to W=H\backslash \widehat{G_{0\sharp}}$ we have $\pi \circ L_\gamma=\pi$ because
$$
(\pi\circ L_\gamma)(x)=[\gamma x]=[x]\in H\backslash \widehat{G_{0\sharp}}.
$$ 
Then, for all $v\in T_x \widehat{G_{0\sharp}}$, we have
\begin{align*}
\dd(f\circ L_\gamma)_x(v)
&= (f\circ L_\gamma)_{*x}(v)\\
&= (f_{*\gamma x}\circ(L_\gamma)_{*x})(v)\\
 &=(\dd f)_{\gamma x}((L_\gamma)_{*x}(v))\\
&= (\pi^* \omega)_{\gamma x}((L_\gamma)_{*x}(v))\\
&= \omega_{[\gamma x]}(\pi_{*\gamma x}((L_\gamma)_{*x}(v)))\\
&=\omega_{[x]}((\pi\circ L_\gamma)_{*x}(v))\\
&=\omega_{[x]}(\pi_{*x}(v))\\
&= (\pi^*\omega)_x(v)\\
&=(\dd f)_x(v).
\end{align*}
As a consequence, $\dd (f\circ L_\gamma)={\rm d}f$ for all $\gamma\in H$. But, since $\widehat{G_{0\sharp}}$ is connected, it follows that $f\circ L_\gamma=f+c(\gamma)$ for some constant $c(\gamma)$ depending only on $\gamma$. Moreover, since $f(e)=0$, we obtain that $c(\gamma)=f(\gamma)$. It follows that for an arbitrary $x\in \widehat{G_{0\sharp}}$ we have $f(\gamma x)=f(\gamma)+ f(x)$. 

\medskip

(2)
Let $\beta$ be a path in $\widehat{G_{0\sharp}}$, joining the identity $e$ to the point  $\gamma\in H$. If we project this path through $\pi$ we obtain a loop $\alpha=\pi \circ \beta$ in $W=H\backslash \widehat{G_{0\sharp}}$. So, by (\ref{integral}), we have
$$
\log \overline m_{H}([\pi\circ\beta])=\int_{\pi\circ\beta} \omega.
$$
Now, the isomorphism  $\pi_1(W)\cong H/H_e$ sends the homotopy class of the loop $\alpha$ into the final point $\beta(1)=\gamma$ of the lifting $\beta$ with $\beta(0)=e$. So 
$\log \overline m_H([\pi\circ\beta])= \log m_H(\gamma)$.

On the other hand,
\begin{align*}
\int_{\pi\circ\beta}\omega&=\int_{[0,1]}(\pi\circ\beta)^*\omega\\
&=\int_{[0,1]} \beta^*\pi^*\omega\\
&=\int_{[0,1]}\beta^*(\dd f)\\
&=\int_{[0,1]} \dd(\beta^*f)\\
&=\int_{[0,1]} \dd(f\circ \beta)\\
&=(f\circ\beta)(1)-(f\circ\beta)(0)\\
&=f(\gamma)-f(e)\\
&=f(\gamma).
\end{align*}
So we have checked that $f(\gamma)=\log m_{H}(\gamma)$ for all $\gamma\in H$. 
\end{proof}

\begin{proof}[Proof of Proposition \ref{TERCER}]

Consider the restriction  $H=\pi_0^{-1}(\overline\Gamma_0)\to \overline\Gamma_0$ of the universal covering $\pi_0\colon \widehat{G_{0\sharp}} \to G_{0\sharp}$. By Proposition \ref{funcionmodulardee2}, it follows that
$$m(k)=m_H(k)= \det \Ad_{\overline\Gamma_0}(e)=1$$ for all $k\in \ker \pi_0\subset H$, where $m\colon \widehat{G_{0\sharp}}\to {\R}^+$ is the map given by  Lemma \ref{CRUCIAL}.
In this way,
the map $m$ passes to the quotient $\widehat{G_{0\sharp}}/\ker \pi_0$, so  we have a map 
$$m^\prime\colon G_{0\sharp}\to\R^+$$ such that
\begin{equation} \label{gammalineal}
m^\prime(\gamma\,  y)=m^\prime(\gamma)\,  m^\prime(y),
\end{equation}
for all $\gamma\in\overline\Gamma_0$, $y\in G_{0\sharp}$.  

Then,  the map $\log m^\prime\colon  G_{0\sharp}\to \R$ is surjective, because $G_{0\sharp}$ is connected and  $m\vert_H=m_H$ is not bounded (the group $H$ is not unimodular).

Let us consider the diagram defining the Blumenthal bundle as in section \ref{blubundle}, that is,
\begin{equation}\label{diagblucubiertas}
\xymatrix{
f^*(G_{0\sharp})\ar[r]^{\overbar f}\ar[d]_\tau&G_{0\sharp}\\
\Gamma_0\backslash f^*(G_{0\sharp})}
\end{equation}
The maps $D=\log m^\prime\circ \overbar f $ and $h^\prime=\log m^\prime \circ h_0$, where $h_0\colon \Aut(\tau)\cong \Gamma_0$, are  respectively  the developing map and the holonomy morphism of an $\R$-Lie foliation  on the (non-connected) manifold $\Gamma_0\backslash f^*G_{0\sharp}$. This manifold is compact by Proposition \ref{fblucompacto}. 

It only remains to show the equivariance, which will follow from the condition \eqref{gammalineal}. In fact,  if $x\in {f^*(G_{0\sharp}})$ and $\gamma\in \Aut \tau\stackrel{h}\cong \Gamma_0$, then
\begin{align*}
D(\gamma  x)&=\log m^\prime(\overbar f (\gamma x))\\
&= \log m^\prime (h(\gamma)\overbar f (x))\\
&=\log(m^\prime(h(\gamma))\cdot m^\prime(\overbar f (x)))\\
&=\log (m^\prime(h(\gamma))+\log(m^\prime(\overbar f (x)))\\
&=h^\prime(\gamma)+D(x).
\end{align*}
because $h(\gamma)\in \overline\Gamma_0$. 

Hence, Tischler theorem applies and allows us to state that $\Gamma_0\backslash f^*(G_{0\sharp})$ fibers over $S^1$.
\end{proof}

\subsection{Lie foliations.}

Remember from Example \ref{liefolprimeraparte} that
the foliation $\F$ on the compact manifold $M$ is a {\em Lie foliation} if it is transversely homogeneous with  transverse model  a Lie group. We can assume that $N=\widehat G_0$ is a connected simply connected Lie group.

In this case the foliation is Riemannian and the developing submersion is a locally trivial bundle with connected fibers. Moreover,  the Blumenthal fiber bundle is identified with $M$. Finally,
Theorem \ref{cohth} reads as follows, as it is well known:

\begin{teo}
Given a $\widehat G_0$-Lie foliation $\F$, with holonomy  $\Gamma_0\subset \widehat G_0$, the base-like cohomology $H(M/{\F})$ is isomorphic to $H_{\overline\Gamma}(\widehat G_0)$.
\end{teo}

On the other hand, our Theorem  \ref{tercerteoremadeuni} shows that if the Lie groups $\widehat G_0$ and $\overline\Gamma_0$ are unimodular then the foliation $\F$ is unimodular. 
El Kazimi Alaoui and Nicolau went further in the study of the unimodularity of Lie foliations and proved the following result.

\begin{teo}\cite[Theorem 1.2.4]{ElKacimi2001}
The $\widehat G_0$-Lie foliation $\F$ is unimodular if and only if the Lie groups $\widehat G_0$ and $\overline\Gamma_0$ are unimodular. 
\end{teo}

Finally, we have proved in
 \cite{Macias2005} the following result, which is now a particular case of our Theorem  \ref{generalizaciondelrendus},   when restricted to  Lie foliations.
 
\begin{teo}
If the Lie foliation $\F$ is not unimodular then either $M$ or the closures of the leaves fiber over $S^1$,
\end{teo}

\begin{example}\label{CARREXAMPLE}What follows is Carri\`ere's example \cite{Carriere1982} cited in the Introduction.
Let $A$ be a matrix in $\SL(2,\Z)$ with  
$\trace A> 2$. We can give a Lie group structure
to $\widehat M=\R^3$ by defining
$$(u,t)\cdot(u',t')=(u+A^tu', t+t').$$
The manifold $M$ will be the quotient of $\widehat M$ by the discrete subgroup $\Z^3$.

Let $\lambda>0$   be an eigenvalue of $A$ with an eigenvector $v=(a,b)\in\R^2$, $\vert v\vert=1$. The affine group $\GA(\R)$ of the real line,  generated by homotheties and translations, can be represented by the matrices
\begin{equation}\label{GAMATRIX}
g=\begin{bmatrix}
\lambda^t  &s \cr
0 & 1 \cr
\end{bmatrix},\quad  s,t\in \R,
\end{equation}
so the map  $\widehat f\colon \R^3  \to \GA(\R)$   given by
$$f(x,y,t)=\begin{bmatrix}
\lambda^t  &ax+by \cr
0 & 1 \cr
\end{bmatrix}$$
is a Lie group morphism. Its kernel (fiber) is the line 
generated by the eigenvectors of  the eigenvalue $1/\lambda$, which induces a Lie flow on $M$. Its leaves are not closed because $\lambda$ is an irrational number, their closures are tori.  The holonomy morphism $h$
will be the restriction of $\widehat f$ to $\pi_1(T^3_A)=\Z^3$.

The closure the image of $h$  is the subgroup $\overline \Gamma \cong \Z\times \R$ of matrices
$$\begin{bmatrix}
\lambda^n  &s \cr
0 & 1 \cr
\end{bmatrix},\quad  n\in\Z,s\in \R.$$
which is abelian, hence unimodular. On the other hand, for $g=(s,t)$ as in \eqref{GAMATRIX}, then $$\Ad(g)=
\begin{bmatrix}
\lambda^t&-s\cr
0&1\cr
\end{bmatrix},$$
so the modular function of $GA(\R)$ is
$$m(s,t)=  \lambda^t.$$
As stated in the proof of Theorem \ref{PARTEUNO}, the map 
$$(\log m\circ \widehat f)(x,y,t)=\log \lambda\, t$$
defines a Lie foliation on $T^3_A$ which is the kernel of the closed $1$-form $\omega = \log\lambda\, dt$. Since the group of periods of this form is the discrete subgroup of $\R$ generated by $\log\lambda$, the foliation is a fibration over $S^1$ with tori as fibers.
\end{example}

\nocite{*}
\bibliographystyle{plain}
\bibliography{biblio_THF}
\end{document}